\documentclass[smallcondensed]{svjour3}  
\usepackage{lineno}
\modulolinenumbers[5]
\smartqed
\usepackage{amssymb,amsmath,latexsym}
\usepackage{enumitem}
\usepackage[margin=2.5cm]{geometry}
\usepackage{caption}

\newcommand{\N}{\mathbb N}
\newcommand{\Z}{\mathbb Z}

\newcommand{\F}{\mathbb F}

\DeclareMathOperator{\wt}{wt}

\makeatletter
\def\old@comma{,}
\catcode`\,=13
\def,{%
  \ifmmode%
    \old@comma\discretionary{}{}{}%
  \else%
    \old@comma%
  \fi%
}
\makeatother

\begin{document}
\title{On the inverses of Kasami and Bracken-Leander exponents}
\author{Lukas K\"olsch} 
\institute{Lukas K\"olsch \at
              University of Rostock, Germany \\
              \email{lukas.koelsch@uni-rostock.de}           
}
\date{Received: date / Accepted: date}
\maketitle
\begin{abstract}We explicitly determine the binary representation of the inverse of all Kasami exponents $K_r=2^{2r}-2^r+1$ modulo $2^n-1$ for all possible values of $n$ and $r$. This includes as an important special case the APN Kasami exponents with $\gcd(r,n)=1$. As a corollary, we determine the algebraic degree of the inverses of the Kasami functions. In particular, we show that the inverse of an APN Kasami function on $\F_{2^n}$ always has algebraic degree $\frac{n+1}{2}$ if $n\equiv 0 \pmod 3$. For $n\not\equiv 0 \pmod 3$ we prove that the algebraic degree is bounded from below by $\frac{n}{3}$. We consider Kasami exponents whose inverses are quadratic exponents or Kasami exponents. We also determine the binary representation of the inverse of the Bracken-Leander exponent $BL_r=2^{2r}+2^r+1$ modulo $2^n-1$ where $n=4r$ and $r$ odd. We show that the algebraic degree of the inverse of the Bracken-Leander function is $\frac{n+2}{2}$.

\end{abstract}

\section{Introduction}

Vectorial Boolean function play a big role in cryptography because of their importance in the construction of S-boxes in block ciphers. To ensure resistance against differential attacks, S-boxes should have low differential uniformity \cite{EC:Nyb}. 

\begin{definition}
	A function $f \colon \F_{2^n} \rightarrow \F_{2^n}$ has differential uniformity $d$, if 
	\begin{equation*}
		d=\max_{a\in \F_{2^n}^*,b \in \F_{2^n}}\left| \{x \colon f(x)+f(x+a)=b\}\right|.
	\end{equation*}
	A function with differential uniformity $2$ is called almost perfect nonlinear (APN) on $\F_{2^n}$.
\end{definition}
Clearly, the differential uniformity is always a multiple of $2$ so APN functions have the lowest possible differential uniformity and give the best protection against differential attacks. In addition to applications in cryptography, APN functions are connected to coding theory and reversed Dickson polynomials \cite{EC:Nyb},  \cite{DDC:CCZ98}, \cite{FFA:HMSY09}. Generally, finding families of APN functions is difficult and remains a challenge in this research area. One of the best understood classes of APN functions are APN monomials (see Table \ref{t:APN}). If the monomial $x \mapsto x^l$, $1 \leq l \leq 2^n-2$ is APN on $\F_{2^n}$, we call $l$ an APN exponent on $\F_{2^n}$. We denote by $\wt(l)$ the binary weight of $l$, i.e. the number of ones in its binary expansion. The binary weight of $l$ is precisely the algebraic degree of the function $x \mapsto x^l$. The algebraic degree is an important cryptographic property since mappings with low algebraic degree are potentially more vulnerable to attacks like higher-order differential attacks \cite{higherorder} or algebraic attacks \cite{algebraic}. It is well known that  the differential uniformity of a permutation is invariant under taking the inverse. In particular, if $l$ is an APN exponent and invertible in $\Z_{2^n-1}$ then the inverse $l^{-1}$ of $l$ modulo $2^n-1$ is an APN exponent as well. We say that $l$ and $l'$ are \emph{cyclomotic equivalent} if $l'\equiv 2^il \pmod{2^n-1}$ for some $i$. Differential uniformity is also invariant under cyclotomic equivalence. If $l$ and $l'$ are cyclotomic equivalent then the binary representation of $l'$ is just a cyclic shift of the binary representation of $l$. Additionally, if $l'$ is invertible modulo $2^n-1$ then $l'^{-1}\equiv 2^{-i}l^{-1} \pmod{2^n-1}$. To fully classify the APN monomials, it is thus necessary to determine the inverse of the known APN exponents (if they exist). It is known that APN exponents are invertible if and only if $n$ is odd (see e.g. \cite[Proposition 9.19.]{carlet}). Determining the explicit binary representations of the inverses of the known APN exponents is thus an interesting problem. The precise binary representations immediately also give the algebraic degree of the function $x \mapsto x^{l^{-1}}$. This has been done for all known APN exponents except for the Kasami exponents (see Table \ref{t:APN}). This paper will close this gap and find an explicit expression for the inverses of all Kasami exponents (if they exist). We will also deal with the non-APN Kasami exponents.

\begin{table}[ht]
\captionsetup{font=scriptsize}
	\centering
	\begin{tabular}{ ||c|c|c|c|c|| } 
	 \hline
		& Exponent & Conditions & Algebraic Degree & Inverse determined in\\
	 \hline  \hline
	 Gold & $2^r+1$ & $\gcd(r,n)=1$  & $2$ & \cite{EC:Nyb}\\
	  &  & $r<n/2$ & & \\
	 \hline
	 Kasami & $2^{2r}-2^r+1$ & $\gcd(r,n)=1$ & $r+1$& This paper\\
		  &  & $r<n/2$ & & \\
	 \hline
	 Welch & $2^t+3$ &  & $3$& \cite{FFA:KyuSud14}\\
	 \hline
	 Niho & $2^t-2^{\frac{t}{2}}-1$ & $t$ even & $\frac{t+2}{2}$ & \cite{Niho},\cite{FFA:KyuSud14}\\
	 & $2^t-2^{\frac{3t+1}{2}}-1$ & $t$ odd& $t+1$ &\\
	 \hline
	 Inverse & $2^{2t}-1$ &  & $n-1$& Obvious\\
	 \hline
	 Dobbertin & $2^{4r}+2^{3r}+2^{2r}+2^r-1$ & $5r=n$  & $r+3$& \cite{FFA:KyuSud14}\\
	 \hline
	\end{tabular}
	\caption{List of known APN exponents over $\F_{2^n}$ with $n=2t+1$ up to inversion and cyclotomic equivalence}
	\label{t:APN}
\end{table}

The objective is thus the following: Find the inverse of $K_r=2^{2r}-2^r+1$ modulo $2^n-1$ for all $r,n$. Compared to other APN exponents, determining the inverses of the Kasami exponents is particularly challenging because they are independent from the field size. APN exponents with this property are called \textit{exceptional APN exponents}. It was shown in \cite{JA:FerMcG11} that Gold and Kasami exponents are the only exceptional APN exponents. Finding the inverses of Gold exponents is relatively easy because of their low binary weight. In contrast, the algebraic degree of the Kasami exponents is unbounded which makes the determination of the inverses much harder. 

In \cite{FFA:KyuSud14}, a method to find the inverse of a fixed exponent $l$ modulo $2^n-1$ for arbitrary $n$ was given. This technique was used to determine the inverses of the second Kasami exponent $K_2=13$. Unfortunately, it is unclear how to use this approach to determine the inverses of all (infinitely many) Kasami exponents. In fact, just determining the binary weight of the inverses of Kasami exponents is mentioned as an open problem in \cite{FFA:KyuSud14}. 

As mentioned earlier, invertible APN exponents do not exist in even dimension. In fact, no APN permutations in even dimension $n \neq 6$ have been discovered yet, so permutations in even dimension with differential uniformity $4$ are of great interest and have been the subject of much research. In this case, it is also interesting to consider monomials with differential uniformity $4$. For a complete list of known families of $4-$differentially uniform permutation monomials in even dimension see Table \ref{t:4uni}. In the next section, we will find the binary representation of the inverses of all invertible Gold exponents and in the fourth section, we will also determine the binary representation of the inverse of the Bracken-Leander function. With this, the inverses of all known exponents that produce monomials with differential uniformity $2$ in odd dimension or $4$ in even dimension are determined.

\begin{table}[ht]
\captionsetup{font=scriptsize}
	\centering
	\begin{tabular}{ ||c|c|c|c|c|| } 
	 \hline
		& Exponent & Conditions & Algebraic Degree &Inverse determined in\\  [0.5ex] 
	 \hline  \hline
	 Gold & $2^r+1$ & $t$ odd, $\gcd(r,n)=2$  & $2$ & \cite{FFA:KyuSud14}, This paper \\
		  &  & $r<n/2$ & & \\
	 \hline
	 Kasami & $2^{2r}-2^r+1$ &$t$ odd, $\gcd(r,n)=2$ & $r+1$& This paper\\
		  &  & $r<n/2$ & & \\
	 \hline
	 Inverse & $2^{n}-2$ &  & $n-1$& Obvious\\
	 \hline
	 Bracken-Leander & $2^{2r}+2^r+1$ & $4r=n$, $r$ odd  & $3$& This paper\\
	 \hline
	\end{tabular}
	\caption{List of exponents yielding $4$ differentially uniform permutations over $\F_{2^n}$ with $n=2t$ up to inversion and cyclotomic equivalence}
	\label{t:4uni}
\end{table}

Our approach in this paper is new and uses as the key tool the \emph{modular add-with-carry approach} that was first formally introduced by Hollmann and Xiang \cite{FFA:HX}. 

\begin{theorem} [\cite{FFA:HX}, Theorem 13]\label{thm:xiang_orig}
	Let $a,s \in \{1,\dots,2^n-2\}$ and $l\in \N$. We denote by $a=(a_{n-1},\dots,a_0)$ and $s=(s_{n-1},\dots,s_0)$ the binary expansions of $a$ and $s$. Let $l=\sum_{j}t_j2^j$ with $t_j \in \Z$. Further, let $t_+=\sum_{j,t_j>0}t_j$ and $t_-=\sum_{j,t_j<0}t_j$.
	The following are equivalent:
	\begin{enumerate}[label=(\alph*)]
		\item $s\equiv l\cdot a \pmod{2^n-1}$ 
		\item There exists a sequence $c=(c_{n-1},\dots,c_{0})$ with $c_i \in \{t_-,t_-+1,\dots,t_+-1\}$ (called the \emph{carry sequence}) such that 
	\begin{equation} \label{eq:fund_orig}
		2c_i-c_{i-1}+s_i=\sum_j t_ja_{i-j}
	\end{equation}
	holds for all $i$. Here, the indices are seen as elements in $\Z_{n}$. 
	\end{enumerate}
	The carry sequence in (b) is unique.
\end{theorem}

\begin{remark}
	Note that the representation $l=\sum_{j}t_j2^j$ with integer coefficients $t_j$ in Theorem~\ref{thm:xiang_orig} is not unique. In fact, this is one of the strengths of this theorem since it makes it possible to choose a representation that has more structure than the (usual) binary representation. This makes a big difference especially for the Kasami exponents. Indeed, the $r$-th Kasami exponent $K_r$ can be written (as it is done usually) as $K_r = 2^{2r}-2^r+1$, i.e. with $t_{2r}=t_0=1$ and $t_r=-1$. This is certainly a much simpler representation than the binary representation that has $r+1$ ones. In the general case, it seems to be desirable to choose a representation such that both $t_+$ and $t_-$ have low absolute value so that the range of possible values for the carry sequence is small.
\end{remark}

The basic idea of finding the inverse of some value $l$ modulo $2^n-1$ is now quite simple: We use Theorem \ref{thm:xiang_orig} and set $s=1$. Then we try to find sequences $a$ and $c$ that satisfy Eq.~\eqref{eq:fund_orig}. While we apply the approach in this paper only to the Gold, Kasami and Bracken-Leander exponents, the idea can in principle be used for arbitrary values of $l$. However, the corresponding sequences $a$ and $c$ are highly dependent on the choice of $l$, so a general treatment seems to be impossible. Still, this approach gives a good framework to find inverses in $\Z_{2^n-1}$.

\section{The Gold exponents}
To ``warm up'' and to illustrate our method using Theorem~\ref{thm:xiang_orig}, we (re-)derive the inverses of the invertible Gold exponents $G_r=2^r+1$ in $\Z_{2^n-1}$. The inverses of the APN Gold exponents (i.e. with the condition $\gcd(r,n)=1$) are explicitly given in \cite{EC:Nyb}. Moreover, the algebraic degree of the inverses of all Gold exponents is known \cite[Theorem 3.7.]{FFA:KyuSud14}. However, the explicit binary expansion of the inverses of the non-APN Gold exponents has not been determined yet. In this section, we apply the add-with-carry approach to find the binary expansion of the inverses of all invertible Gold exponents.  

Applied to the Gold exponent, Theorem~\ref{thm:xiang_orig} yields the following.

\begin{theorem} \label{thm:xiang_gold}
	Let $a,s \in \{1,\dots,2^n-2\}$ and $G_r=2^r+1$ be the $r$-th Gold exponent. We denote by $a=(a_{n-1},\dots,a_0)$ and $s=(s_{n-1},\dots,s_0)$ the binary expansions of $a$ and $s$. 
	The following are equivalent:
	\begin{enumerate}[label=(\alph*)]
		\item $s\equiv G_r\cdot a \pmod{2^n-1}$ 
		\item There exists a carry sequence $c=(c_{n-1},\dots,c_{0})$ with $c_i \in \{0,1\}$ such that 
	\begin{equation} \label{eq:fund_gold}
		2c_i-c_{i-1}+s_i=a_{i-r}+a_i
	\end{equation}
	holds for all $i$. Here, the indices are seen as elements in $\Z_{n}$. 
	\end{enumerate}
	The carry sequence in (b) is unique.
\end{theorem}

The following lemma characterizes all invertible Gold exponents. 

\begin{lemma}[{e.g. \cite[Lemma 11.1.]{ff_mceliece}}]\label{lem:goldinv}
	Let $r$ and $n$ be positive integers. The Gold exponent $G_r=2^r+1$ is invertible in $\Z_{2^n-1}$ if and only if $\frac{n}{\gcd(n,r)}$ is odd.
\end{lemma}
\subsection{The APN Gold exponents}
We first deal with the APN Gold exponents $G_r=2^r+1$ over $\F_{2^n}$ with $\gcd(r,n)=1$. We will use some notation from~\cite{stickel}, where the modular add-with-carry approach was used to find the Walsh support of the Kasami functions. In particular, we will use the notion of $r$-ordered sequences.


Since $\gcd(r,n)=1$ we can reorder the sequences $a$ and $c$ in Theorem~\ref{thm:xiang_gold} in the following way:
\begin{equation*}
	a_0,a_{-r},a_{-2r},\dots,a_{-(n-1)r} \text{ and }c_0,c_{-r},c_{-2r},\dots,c_{-(n-1)r}.
\end{equation*}
Here, we view again the indices as elements in $\Z_n$. This ordering is technically a decimation of the sequence by $-r$. Since we will be using this ordering a lot, we will call it the \emph{$r$-ordering of a sequence} and also denote these sequences by 
\begin{equation*}
	a_0,a_1,a_{2},\dots,a_{n-1} \text{ and }c_0,c_{1},c_{2},\dots,c_{n-1},
\end{equation*}
where we will always make sure to specify whether we use the regular ordering or the $r$-ordering.

By Lemma~\ref{lem:goldinv}, $G_r$ is invertible if and only if $n$ is odd. We denote by $e$ the least positive residue of the inverse of $r$ modulo $n$. 
Using $r$-ordered sequences, the key equation in Theorem \ref{thm:xiang_gold} takes on the following simpler form:
\begin{theorem} \label{thm:gcd1_gold}
	Let $n \in \N$, $a\in \{1,\dots,2^n-2\}$ and $G_r$ be the $r$-th Gold exponent with $\gcd(r,n)=1$. Let $e$ be the least positive residue of the inverse of $r$ modulo $n$ and $(a_{0},\dots,a_{n-1})$ be the $r$-ordered sequence of the binary representation of $a$, i.e. $a\equiv \sum_{i=0}^{n-1}a_i2^{-ir} \pmod{2^n-1}$. 	The following are equivalent:
	\begin{enumerate}[label=(\alph*)]
		\item $a$ is the inverse of $G_r$ modulo $2^n-1$.
			\item There exists an $r$-ordered carry sequence $c=(c_0,c_1,\dots,c_{n-1})$ with $c_i \in \{0,1\}$  such that 
	\begin{align} 
		2c_0-c_{e}+1&=a_{1}+a_{0} \label{eq:gcd1_1_gold} \\
		2c_i-c_{i+e}&=a_{i+1}+a_{i} \label{eq:gcd1_2_gold}
	\end{align}
		holds for all $i\in \Z_{n}$, $i \neq 0$. 
	\end{enumerate}
	The carry sequence in (b) is unique.
\end{theorem}

Now, we can use Theorem~\ref{thm:gcd1_gold} to give a simple alternative proof for the inverse of the APN Gold exponents. As pointed out earlier, the idea is to guess the structure of the carry sequence from examples for low $n$ and then compute the inverse from the carry sequence. We will make a detailed example to give an intuition for this process. Note that the case of APN Gold functions is easier than other cases (especially the Kasami cases in later section), but the approach will always remain the same. 

\begin{example}
	Let $n=7$, $r=3$ and consider the invertible APN Gold exponent $G_3 = 2^3+1=9$. We have $3\cdot 5 \equiv 1 \pmod 7$, so $e = 5$. The inverse of $9$ modulo $2^7-1$ is $113 = 2^6+2^5+2^4+2^0$, and the binary sequence of $113$ in regular ordering is $(a_6,a_5,\dots,a_0) = (1,1,1,0,0,0,1)$ and in $r$-ordering $(a_0,\dots,a_6)=(1,1,0,1,0,1,0)$. We have $a_1=a_0=1$ so by Eq.~\eqref{eq:gcd1_1_gold} necessarily $c_0 = c_5 = 1$. If $i \neq 0$ we have $a_i+a_{i+1} = 1$ and thus by Eq.~\eqref{eq:gcd1_2_gold} $c_i = c_{i+5} = 1$. We conclude that the carry sequence consists exclusively of ones. 
\end{example}
	From this example (and possibly other examples for low $n$) we guess that the carry sequence always consists only of ones. Let us now consider the Eq.s~\eqref{eq:gcd1_1_gold} and~\eqref{eq:gcd1_2_gold} for $c_0=c_1=\dots=c_{n-1}=1$ and arbitrary $n$. It necessarily yields $a_1 = a_0 = 1$ and $a_{i+1} + a_{i} = 1$. From this, we immediately conclude that the $r$-ordered sequence of the inverse $G_r^{-1}$ is $(a_0,\dots,a_{n-1})=(1,1,0,1,0,1,\dots,0,1,0)$. Since both Eq.s~\eqref{eq:gcd1_1_gold} and~\eqref{eq:gcd1_2_gold} are satisfied, this must be the $r$-ordered sequence of the inverse of the Gold exponent. We have thus proven the following:
\begin{proposition}[{\cite[Proposition 5]{EC:Nyb}}] \label{prop:goldgcd1}
	Let $G_r=2^r+1$ with $\gcd(r,n)=1$ and $n$ odd. Then $G_r$ is invertible in  $\Z_{2^n-1}$ and the least positive residue of its inverse is 
	\begin{equation*}
		G_r^{-1}=\sum_{i=0}^{\frac{n-1}{2}}2^{2ir}. 
	\end{equation*}
	In particular, $\wt(G_r^{-1})=\frac{n+1}{2}$, so the algebraic degree of $x \mapsto x^{G_r^{-1}}$ over $\F_{2^n}$ is $\frac{n+1}{2}$.
\end{proposition}
\begin{proof}
	By the considerations above, the $r$-ordered sequence of the inverse $G_r^{-1}$ is $(a_0,\dots,a_{n-1})=(1,1,0,1,0,1,\dots,0,1,0)$. We conclude
	\begin{equation*} 
		G_r^{-1}\equiv 1+\sum_{\substack{i\in \{0,\dots,n-1\} \\i \text{ odd}}}2^{-ir} \equiv \sum_{\substack{i\in \{0,\dots,n-1\} \\i \text{ even}}}2^{ir}\equiv \sum_{i=0}^{\frac{n-1}{2}}2^{2ir} \pmod{2^n-1}.
	\end{equation*}
	\qed
\end{proof}

The main takeaway from the example is that the carry sequence has a simpler structure than the sequence $(a_0,\dots,a_{n-1})$ of the inverse. This observation will also hold for all other exponents considered in this paper. Indeed, while it is still possible to discern the structure of the APN Gold exponents with relative ease without looking at the carry sequence, this will be close to impossible in the case of the Kasami exponents.

\subsection{The non-APN Gold exponents}
We now deal with the more general case of Gold exponents $G_r$ with $\gcd(r,n)>1$.

Since $\gcd(n,r)>1$, we cannot use the $r$-ordering of sequences that we used in the previous proposition. We expand the concept in a natural way.

\begin{definition}\label{def:rmatrix}
	Let $a = (a_{n-1},\dots,a_{0})$ be a sequence of integers and $r$ be a positive integer. Set $d=\gcd(n,r)$. We define the associated $(d \times \frac{n}{d})$-matrix $M_{a,r}$ by
	\begin{equation*}
		M_{a,r}=\begin{pmatrix}
			a_0&a_{-r}&a_{-2r}&\hdots&a_{-(\frac{n}{d}-1)r} \\
			a_1&a_{1-r}&a_{1-2r}&\hdots&a_{1-(\frac{n}{d}-1)r} \\
			\vdots& & & &\vdots \\
			a_{d-1}&a_{d-1-r} &a_{d-1-2r}& \hdots & a_{d-1-(\frac{n}{d}-1)r}
		\end{pmatrix},
	\end{equation*}
	where the indices are seen as elements in $\Z_n$. We call $M_{a,r}$ the \emph{$r$-matrix of $a$}. If $a$ is a binary sequence, then we call $M_{a,r}$ also the $r$-matrix of the corresponding element in $\Z_{2^n-1}$ or $\{0,1,\dots,2^n-2\}$.  
\end{definition}

Since the $r$-matrices are constructed from sequences, we use the slightly unusual convention of indexing from $0$, i.e. the first row/column will be called row/column $0$. With this convention, the $r$-ordered sequences considered in the previous section are just a special case of $r$-matrices with only one row. Again in accordance to the notation used earlier, we denote by $e$ the least positive residue of the inverse of $\frac{r}{\gcd(n,r)}$ modulo $\frac{n}{\gcd(n,r)}$. 

We now use $r$-matrices to rephrase Theorem~\ref{thm:xiang_gold}.

\begin{theorem} \label{thm:gcd>1_gold}
	Let $a \in \{1,\dots,2^n-2\}$, $n \in N$ and $G_r$ be the $r$-th Gold exponent with $\gcd(r,n)=d$ and $e$ be the least positive residue of the inverse of $\frac{r}{d}$ modulo $\frac{n}{d}$. Moreover, let \begin{equation*}
		M_{a,r}=\begin{pmatrix}
			a_{0,0}&a_{0,1}&a_{0,2}&\hdots&a_{0,\frac{n}{d}-1} \\
			a_{1,0}&a_{1,1}&a_{1,2}&\hdots&a_{1,\frac{n}{d}-1} \\
			\vdots& & & &\vdots \\
			a_{d-1,0}&a_{d-1,1} &a_{d-1,2}& \hdots & a_{d-1,\frac{n}{d}-1}
		\end{pmatrix}
	\end{equation*} be the $r$-matrix of $a$, i.e. $a\equiv \sum_{i=0}^{d-1}\sum_{j=0}^{\frac{n}{d}-1}a_{i,j}2^{i-jr} \pmod{2^n-1}$. 	The following are equivalent:
	\begin{enumerate}[label=(\alph*)]
		\item $a$ is the inverse of $G_r$ modulo $2^n-1$.
			\item There exists an $r$-matrix for the carry sequence $c$ of the form 
			\begin{equation*}
		M_{c,r}=\begin{pmatrix}
			c_{0,0}&c_{0,1}&c_{0,2}&\hdots&c_{0,\frac{n}{d}-1} \\
			c_{1,0}&c_{1,1}&c_{1,2}&\hdots&c_{1,\frac{n}{d}-1} \\
			\vdots& & & &\vdots \\
			c_{d-1,0}&c_{d-1,1} &c_{d-1,2}& \hdots & c_{d-1,\frac{n}{d}-1}
		\end{pmatrix}
	\end{equation*} 
			with $c_{i,j} \in \{0,1\}$  such that the following equations hold:
	\begin{align} 
		2c_{0,0}-c_{d-1,e}+1&=a_{0,1}+a_{0,0} \label{eq:gcds_1_gold} \\
		2c_{0,j}-c_{d-1,j+e}&=a_{0,j+1}+a_{0,j}  \text{ for all }j\in \{1,\dots,\frac{n}{d}-1\} \label{eq:gcds_2_gold}\\
		2c_{i,j}-c_{i-1,j}&=a_{i,j+1}+a_{i,j} \text{ for all }i \in \{1,\dots,d-1\}, j \in \{0,\dots,\frac{n}{d}-1\}. \label{eq:gcds_3_gold}
	\end{align}
	\end{enumerate}
	The carry sequence (and thus its associated $r$-matrix) in (b) is unique.
\end{theorem}
\begin{proof}
	The Theorem follows immediately from Theorem \ref{thm:xiang_gold} and the definition of the $r$-matrix. The predecessor of the values $c_{-k_1r}$ is determined as follows: Observe that $c_{-k_1r-1}=c_{\gcd(n,r)-1-k_2r}$ if and only if $-k_1r-1 \equiv \gcd(n,r)-1-k_2r \pmod n$, which is equivalent to $-(k_1-k_2)\frac{r}{\gcd(n,r)} \equiv 1 \pmod{\frac{n}{\gcd(n,r)}}$, so the predecessor of $c_{-k_1r}$ is $c_{\gcd(n,r)-1-k_2r}$ with $k_2=k_1+e$. 
	\qed
\end{proof}

With Theorem~\ref{thm:gcd>1_gold}, we can give the explicit binary representation of all inverses of Gold exponents.

\begin{proposition} \label{prop:gcds2_gold}
	Let $n\in \N$ and $G_r$ be the $r$-th Gold exponent with $\gcd(r,n)=d>1$ and $\frac{n}{d}$ odd. Let $e$ be the least positive residue of the inverse of $\frac{r}{d}$. Then $G_r^{-1}\equiv \sum_{i=0}^{d-1}\sum_{j=0}^{\frac{n}{d}-1}a_{i,j}2^{i-jr} \pmod{2^n-1}$ where the values $a_{i,j}$ are the entries of the $(d \times \frac{n}{d})$-matrix
	\begin{equation*}M_{a,r}=
		\begin{pmatrix}
			0 & 0 & 1 & 0 & 1 & 0 & \dots & 0 & 1\\
			0 & 0 & 1 & 0 & 1 & 0 & \dots & 0 & 1\\
			\vdots \\
			0 & 0 & 1 & 0 & 1 & 0 & \dots & 0 & 1\\
			1 & 1 & 0 & 1 & 0 & 1 & \dots & 1 & 0\\
		\end{pmatrix}.
	\end{equation*}
	In particular, $\wt(G_r^{-1})=\frac{n-d+2}{2}$.
\end{proposition}
\begin{proof}
	The $r$-matrix of the corresponding carry sequence is the $(d \times \frac{n}{d})$-matrix
		\begin{equation*}M_{c,r}=(c_{i,j})=
		\begin{pmatrix}
0 & 1 & 1 & 1 & \dots & 1 \\
0 & 1 & 1 & 1 & \dots & 1 \\
\vdots &  &  &  & & \vdots \\
0 & 1 & 1 & 1 & \dots & 1 \\
1 & 1 & 1 & 1 & \dots & 1 \\
		\end{pmatrix}.
	\end{equation*}	
	We now just have to verify Eq.s~\eqref{eq:gcds_1_gold} to~\eqref{eq:gcds_3_gold}. For Eq.~\eqref{eq:gcds_1_gold}, we get $2c_{0,0}-c_{d-1,e}+1=a_{0,1}+a_{0,0}=0$. For Eq.~\eqref{eq:gcds_2_gold}, we have $2c_{0,j}-c_{d-1,j+e}=a_{0,j+1}+a_{0,j}=1$ for all values of $j>0$. For Eq.~\eqref{eq:gcds_3_gold}, we have $2c_{i,j}-c_{i-1,j}=a_{i,j+1}+a_{i,j}=0$ if $j=0$ and $0<i <d-1$, $2c_{i,j}-c_{i-1,j}=a_{i,j+1}+a_{i,j}=2$ if $j=0$ and $i = d-1$ and $2c_{i,j}-c_{i-1,j}=a_{i,j+1}+a_{i,j}=1$ in all other possible cases. Thus, all equations are satisfied. 
	
	The value of $\wt(G_r^{-1})$ can be determined easily by counting the ones in the matrix $M_{a,r}$.
	\qed
	\end{proof}
	Recall that all invertible Gold exponents satisfy $\frac{n}{\gcd(r,n)}$ odd by Lemma~\ref{lem:goldinv}, so Propositions~\ref{prop:goldgcd1} and~\ref{prop:gcds2_gold} cover all invertible Gold exponents.
	
Note that the $\gcd(r,n)=1$ case can even be recovered as a special case from Proposition~\ref{prop:gcds2_gold}. Indeed, the last row of the $r$-matrices $M_{a,r}$ and $M_{c,r}$ are precisely the $r$-sequences we saw in the $\gcd(r,n)=1$ case in Proposition~\ref{prop:goldgcd1}.

\section{The Kasami exponents}
Applied to the Kasami exponent $K_r=2^{2r}-2^r+1$, Theorem~\ref{thm:xiang_orig} yields the following. 
\begin{theorem} \label{thm:xiang}
	Let $a,s \in \{1,\dots,2^n-2\}$ and $K_r$ be the $r$-th Kasami exponent. We denote by $a=(a_{n-1},\dots,a_0)$ and $s=(s_{n-1},\dots,s_0)$ the binary expansions of $a$ and $s$. 
	The following are equivalent:
	\begin{enumerate}[label=(\alph*)]
		\item $s\equiv K_r\cdot a \pmod{2^n-1}$ 
		\item There exists a carry sequence $c=(c_{n-1},\dots,c_{0})$ with $c_i \in \{-1,0,1\}$ such that 
	\begin{equation} \label{eq:fund}
		2c_i-c_{i-1}+s_i=a_{i-2r}-a_{i-r}+a_i
	\end{equation}
	holds for all $i$. Here, the indices are seen as elements in $\Z_{n}$. 
	\end{enumerate}
	The carry sequence in (b) is unique.
\end{theorem}

We extend the definition of the weight of a sequence to the sum of all of its elements. For binary sequences, this corresponds exactly to its binary weight. In particular, this allows us to talk about the weight of the carry sequence. Using this convention, the following Lemma gives an additional condition on the carry sequence. 

\begin{lemma}[\cite{FFA:HX}, Lemma 5] \label{lem:carry}
	With the notation of Theorem \ref{thm:xiang}, we have the following:
	\begin{enumerate}[label=(\alph*)]
		\item $c_i+c_{i-r} \in \{-1,0,1\}$. In particular, $|\wt(c)| \leq \frac{n}{2}$.
		\item $\wt(c)+\wt(s)=\wt(a)$. In particular, for $s=1$ we have $\wt(c)=\wt(a)-1$.
	\end{enumerate}
\end{lemma}


The following Proposition shows when a Kasami exponent is invertible modulo $2^n-1$.

\begin{proposition}[\cite{FFA:KyuSud14}, Lemma 3.8] \label{prop:invexist}
	Let $n$ be a positive integer and $K_r=2^{2r}-2^r+1$ be the $r$-th Kasami exponent. $K_r$ is invertible modulo $2^n-1$ if and only if one of the following cases occurs:
	\begin{itemize}
		\item $\frac{n}{\gcd(r,n)}$ is odd,
		\item $\frac{n}{\gcd(r,n)}$ is even, $r$ is even and $\gcd(r,n)=\gcd(3r,n)$.
	\end{itemize}
\end{proposition}

We first deal with the case $\gcd(r,n)=1$, then with the case $\frac{n}{\gcd(r,n)}$ odd and finally with the case $\frac{n}{\gcd(r,n)}$ even. Technically, the case $\gcd(r,n)=1$ is included in the case $\frac{n}{\gcd(r,n)}$ odd. However, we single out this case for two reasons: Firstly, it is particularly interesting since those Kasami exponents are precisely the APN exponents. Secondly, the case $\frac{n}{\gcd(r,n)}$ odd is very technical, but can be described much easier by applying the results for the special case $\gcd(r,n)=1$.

\subsection{The case $\gcd(r,n)=1$}
We first deal with the APN Kasami exponents $K_r=2^{2r}-2^r+1$ over $\F_{2^n}$ with $\gcd(r,n)=1$. For this case, the modular add-with-carry approach was already applied in \cite{stickel} to determine the Walsh support of the Kasami function $x \mapsto x^{K_r}$. 

By Proposition~\ref{prop:invexist}, $K_r$ is invertible if and only if $n$ is odd. We denote by $e$ the least positive residue of the inverse of $r$ modulo $n$. Observe that $K_r$ and $K_{n-r}$ are cyclotomic equivalent exponents on $\F_{2^n}$. Indeed, $(2^{2(n-r)}-2^{n-r}+1)2^{2r} \equiv 2^{2r}-2^r+1 \pmod{2^n-1}$. Then $K_r^{-1} \equiv 2^{-2r}K_{n-r}^{-1}\pmod{2^n-1}$, so it suffices to determine the inverse of one of these two values. Since $n$ is odd, we can thus assume without loss of generality that $e$ is odd. 

Since $\gcd(r,n)=1$ we can reorder the sequences $a$ and $c$ in Theorem~\ref{thm:xiang} using the $r$-sequences introduced in the previous section.
Using $r$-ordered sequences, the key equation for the Kasami exponents in Theorem \ref{thm:xiang} takes on the following form:
\begin{theorem} \label{thm:gcd1}
	Let $n \in \N$, $a\in \{1,\dots,2^n-2\}$ and $K_r$ be the $r$-th Kasami exponent with $\gcd(r,n)=1$. Let $(a_{0},\dots,a_{n-1})$ be the $r$-ordered sequence of the binary representation of $a$, i.e. $a\equiv \sum_{i=0}^{n-1}a_i2^{-ir} \pmod{2^n-1}$. 	The following are equivalent:
	\begin{enumerate}[label=(\alph*)]
		\item $a$ is the inverse of $K_r$ modulo $2^n-1$.
			\item There exists an $r$-ordered carry sequence $c=(c_0,c_1,\dots,c_{n-1})$ with $c_i \in \{-1,0,1\}$  such that 
	\begin{align} 
		2c_0-c_{e}+1&=a_{2}-a_{1}+a_{0} \label{eq:gcd1_1} \\
		2c_i-c_{i+e}&=a_{i+2}-a_{i+1}+a_{i} \label{eq:gcd1_2}
	\end{align}
		holds for all $i\in \Z_{n}$, $i \neq 0$. 
	\end{enumerate}
	The carry sequence in (b) is unique.
\end{theorem}

Experimental results show that the inverses of the APN Kasami exponents often have binary weight $\frac{n+1}{2}$. In this case, Lemma \ref{lem:carry} immediately shows that the $r$-ordered carry sequence has weight $\frac{n-1}{2}$ and must be a cyclic shift of the sequence $(0,0,1,0,1,\dots,0,1)$. Since the carry sequence  of the inverse uniquely determines the inverse, these cases can then be solved with comparatively little effort. 

In this section, we will always use $r$-ordered sequences to represent inverses of Kasami exponents because this notation makes the description much easier. Consequently, the inverses will be written in the form $K_r^{-1}\equiv \sum_{i=0}^{n-1}a_i2^{-ir}\pmod {2^n-1}$ for a sequence $a=(a_0,\dots,a_{n-1})$. Of course, a translation into the more standard binary representation is easy by reordering the sequence $a$, i.e.  $K_r^{-1}\equiv \sum_{i=0}^{n-1}a_{-ie}2^{i}\pmod {2^n-1}$ (recall that $e$ denotes the inverse of $r$ modulo $n$).

\begin{proposition} \label{prop:gcd1_1}
	Let $n$ odd, $K_r$ be the $r$-th Kasami exponent with $\gcd(r,n)=1$. Let $e$ be the least positive residue of the inverse of $r$ modulo $n$.  The inverse of $K_r$ modulo $2^n-1$ is 
			\begin{equation*}
			K_r^{-1} \equiv \sum_{i=0}^{n-1}a_i2^{-ir} \pmod {2^n-1},
		\end{equation*}
		where $a=(a_0,\dots,a_{n-1})$ is determined as follows:
	\begin{itemize}
			\item If $e=6k+1$, then $a=(1,x,y)$ and $x=(1,0,1,0,\dots,1,0,1,0)$ is a sequence of length $n-e$ and $y=(1,1,1,0,0,0,1,1,1,0,0,0,\dots,1,1,1,0,0,0)$ is a sequence of length $6k$. 
		\item If $e=6k+5$, then $a=(0,x,1,1,y)$ and $x=(0,1,0,1,\dots,0,1,0,1)$ is a sequence of length $n-e+2$ and $y=(0,0,0,1,1,1,0,0,0,1,1,1,\dots,0,0,0,1,1,1)$ is a sequence of length $6k$. 
	\end{itemize}
	In both cases, we have $\wt(K_r^{-1})=\frac{n+1}{2}$.
\end{proposition}
\begin{proof}
Let $e=6k+1$. Set $c = (0,1,0,1,\dots ,0,1,0,1,0)$, i.e. $c_i=0$ if $i$ is even, and $c_i=1$ otherwise. We show that $a$ and $c$ satisfy the conditions in Theorem \ref{thm:gcd1}. Eq.~\eqref{eq:gcd1_1} can be easily verified. For Eq.~\eqref{eq:gcd1_2}, we have the following:\\
\emph{Case 1:} $i$ odd, $i+e<n$:
We have $c_i=1$, $c_{i+e}=0$, $a_i=a_{i+2}=1$ and $a_{i+1}=0$. \\
\emph{Case 2:} $i$ even, $i+e<n$:
We have $c_i=0$, $c_{i+e}=1$, $a_i=a_{i+2}=0$ and $a_{i+1}=1$.  \\
\emph{Case 3:} $i$ odd, $i+e\geq n$:
We have $c_i=c_{i+e}=1$.  Depending on the value of $i$, the triple $(a_i,a_{i+1},a_{i+2})$ takes on the values $(1,0,0)$, $(0,0,1)$ or $(1,1,1)$.  \\
\emph{Case 4:} $i$ even, $i+e\geq n$:
We have $c_i=c_{i+e}=0$.  Depending on the value of $i$, the triple $(a_i,a_{i+1},a_{i+2})$ takes on the values $(0,0,0)$, $(0,1,1)$ or $(1,1,0)$.   \\
So Eq.~\eqref{eq:gcd1_2} holds for all $i$.\\

Now let $e=6k+5$. Set $c=(0,0,1,0,1,0,1,0,1,\dots,0,1)$, i.e. $c_i=0$ if  $i=0$ or $i$ odd and $c_i=1$ otherwise. We again show that Equations~\eqref{eq:gcd1_1} and~\eqref{eq:gcd1_2} are satisfied.  Observe that Eq.~\eqref{eq:gcd1_1} holds. We check the following cases of Eq.~\eqref{eq:gcd1_2} for $i>0$:\\
\emph{Case 1:} $i$ odd, $i+e \leq n$:
We have $c_i=0$, $c_{i+e}=1$, $a_i=a_{i+2}=0$ and $a_{i+1}=1$. \\
\emph{Case 2:} $i$ even, $i+e\leq n$:
We have $c_i=1$, $c_{i+e}=0$, $a_i=a_{i+2}=1$ and $a_{i+1}=0$.  \\
\emph{Case 3:} $i$ odd, $i+e>n$:
We have $c_i=c_{i+e}=0$.  Depending on the value of $i$, the triple $(a_i,a_{i+1},a_{i+2})$ takes on the values $(0,0,0)$, $(0,1,1)$ or $(1,1,0)$.  \\
\emph{Case 4:} $i$ even, $i+e> n$:
We have $c_i=c_{i+e}=1$.  Depending on the value of $i$, the triple $(a_i,a_{i+1},a_{i+2})$ takes on the values $(1,0,0)$, $(0,0,1)$ or $(1,1,1)$.   \\
So Eq.~\eqref{eq:gcd1_2} holds for all $i$.
\qed
\end{proof}
The Kasami APN functions and their inverses are also almost bent functions. It is known that the algebraic degree of almost bent functions is at most $\frac{n+1}{2}$ \cite{DDC:CCZ98}. We have shown that the inverses of the Kasami APN functions defined by the exponents considered in Proposition \ref{prop:gcd1_1} attain this bound. 

The only case left to check is $e=6k+3$ (recall that we could assume $e$ odd without loss of generality). This case is a lot more involved and has to be divided into several subcases. The key difference to the cases considered above is that $\wt(K_r^{-1})<\frac{n+1}{2}$ for $e=6k+3$, so finding the correct carry sequence is more complicated. However, the strategy of the proof remains the same: Based on experimental results, we guess a carry sequence that then determines the inverse.

\begin{proposition} \label{prop:gcd1_2}
	Let $n$ odd, $K_r$ be the $r$-th Kasami exponent with $\gcd(r,n)=1$. Let $e=6k+3$ be the least positive residue of the inverse of $r$ modulo $n$. Define $s,t \in \N$ by $n=se+t$ with $0\leq t <e$. Further, let $x_1=(0,0,0,1,1,1)$, $x_2=(0,1,1,1,0,0)$ be sequences of length $6$ and 
	\begin{align*}
		x&=(0,1,1,\underbrace{x_1,\dots,x_1}_{\text{k-times}},0,0,0,\underbrace{x_2,\dots,x_2}_{\text{k-times}}) \\
		y&=(0,0,0,\underbrace{x_2,\dots,x_2}_{\text{k-times}},0,1,1,\underbrace{x_1,\dots,x_1}_{\text{k-times}})
	\end{align*}
	be sequences of length $2e$. Then
		\begin{equation*}
			K_r^{-1} \equiv \sum_{i=0}^{n-1}a_i2^{-ir} \pmod {2^n-1},
		\end{equation*}
		where $a=(a_0,\dots,a_{n-1})$ is determined as follows:
	\begin{enumerate}[label=(\alph*)]
				\item If $t = 6u+1$ then
		\begin{equation*}
		a=(\underbrace{x_1,\dots,x_1}_{u\text{-times}},\underbrace{y,\dots,y}_{(s-2)/2\text{-times}},0,0,0,\underbrace{x_2,\dots,x_2}_{\text{k-times}},0,1,1,\underbrace{x_1,\dots,x_1}_{\text{u-times}},z,0) +(1,0,\dots,0),
		\end{equation*}
	where $z=(0,1,0,1,\dots,0,1,0,1)$ is a sequence of length $e-3-6u$.
		
		\item If $t = 6u+2$ then
		\begin{equation*}
		a=(0,1,\underbrace{x_1,\dots,x_1}_{u\text{-times}},\underbrace{y,\dots,y}_{(s-1)/2\text{-times}},0,0,z,1,z_2),
		\end{equation*}
		where $z=(1,0,1,0\dots,1,0,1,0)$ is a sequence of length $6u$ and $z_2=(x_2,\underbrace{x_1,\dots,x_1}_{\substack{(e-6u-9)/6 \\ \text{-times}}})$ is a sequence of length $e-6u-3$.
		
			\item If $t = 6u+4$ then
		\begin{equation*}
		a=(0,0,0,\underbrace{x_2,\dots,x_2}_{\text{u-times}},\underbrace{x,\dots,x}_{(s-1)/2\text{-times}},0,1,1,0,z,1,0,1,1,1,\underbrace{x_1,\dots,x_1}_{{\substack{(e-6u-9)/6 \\ \text{-times}}}},0)
		\end{equation*}
		where $z=(0,1,0,1,\dots,0,1,0,1)$ is a sequence of length $6u$. 
		
	\item If $t = 6u+5$ then
		\begin{equation*}
		a=(0,1,1,0,0,\underbrace{x_2,\dots,x_2}_{\text{u-times}},\underbrace{x,\dots,x}_{(s-2)/2\text{-times}},0,1,1,\underbrace{x_1,\dots,x_1}_{k\text{-times}},0,\underbrace{x_1,\dots,x_1}_{u\text{-times}},z)
		\end{equation*}
		where $z=(0,1,0,1,\dots,0,1,0,1)$ is a sequence of length $e-1-6u$. 

	\end{enumerate}
	In the cases (a) and (d) we have $\wt(K_r^{-1})=\frac{n-s+1}{2}$ and in the cases (b) and (c) $\wt(K_r^{-1})=\frac{n-s}{2}$.
\end{proposition}
\begin{proof}
	For all 4 cases, we explicitly give the carry sequence $c$ (in $r$-ordering) and check that Eq.~\eqref{eq:gcd1_1} and~\eqref{eq:gcd1_2} are satisfied. The carry sequences for all cases are quite similar and are composed of the same ``building blocks''. The verification is simple but tedious, so we will show the correctness of the first case in detail and for the other cases we will just state the carry sequence and omit the verification. We define the auxiliary sequences $s_1=(0,1,0,1,\dots,0,1)$ of length $6u$ and $s_2=(0,0,1,0,1,0,1,0,1,\dots,0,1)$ of length $e=6k+3$.  \\
	\emph{Case (a)}: Set 
	\begin{equation*}
		c=(s_1,\underbrace{s_2,\dots,s_2}_{s\text{-times}},0).
	\end{equation*}
	Eq.~\eqref{eq:gcd1_1} can be easily verified. For Eq.~\eqref{eq:gcd1_2} we have to distinguish (many) different cases depending on the value of $i$. We go through each block in the sequence $a$.\\
	\underline{Case a.1:} $i>0$ is in the first block of $x_1$'s. If $i$ is even then we have $c_i=c_{i+e}=0$ and $(a_i,a_{i+1},a_{i+2})\in \{(0,1,1),(1,1,0),(0,0,0)\}$. If $i$ is odd then $c_i=c_{i+e}=1$ and $(a_i,a_{i+1},a_{i+2})\in \{(0,0,1),(1,1,1),(1,0,0)\}$.\\
	\underline{Case a.2:} $i$ is in the block of $y$'s. Let $i=6u+q$. If $q \equiv  1,2,e+1,e+2 \pmod {2e}$ then $c_i=c_{i+e}=0$. In these first two cases we have $(a_i,a_{i+1},a_{i+2})=(0,0,0)$ and in the latter two $(a_i,a_{i+1},a_{i+2})=(0,1,1)$ and $(a_i,a_{i+1},a_{i+2})=(1,1,0)$, respectively. Let $q_1$ be the least positive residue of $q$ modulo $2e$.  If $3 \leq q_1 \leq e$ and $q_1$ odd we have  $c_i=c_{i+e}=0$ and $(a_i,a_{i+1},a_{i+2})\in \{(0,1,1),(1,1,0),(0,0,0)\}$. If $3 \leq q_1 \leq e$ and $q_1$ even, we have $c_i=c_{i+e}=1$ and $(a_i,a_{i+1},a_{i+2})\in \{(0,0,1),(1,1,1),(1,0,0)\}$. If $q_1>e+2$ and $q_1$ odd we have $c_i=c_{i+e}=1$ and $(a_i,a_{i+1},a_{i+2})\in \{(0,0,1),(1,1,1),(1,0,0)\}$ and if $q_1>e+2$ and $q_1$ even we have $c_i=c_{i+e}=0$ and $(a_i,a_{i+1},a_{i+2})\in \{(0,1,1),(1,1,0),(0,0,0)\}$.\\
	\underline{Case a.3:} $i$ is in the position of the three zeros after the block of $y$'s. For the first two zeros (i.e. $i=6u+e(s-2)$ and $i=6u+e(s-2)+1$) we have $c_i=c_{i+e}=0$ and $(a_i,a_{i+1},a_{i+2})=(0,0,0)$. For $i=6u+e(s-2)+2$ we have $c_i=c_{i+e}=1$ and $(a_i,a_{i+1},a_{i+2})=(0,0,1)$. \\
\underline{Case a.4:} $i$ is in the block of $x_2$'s, i.e. $i\in \{6u+e(s-2)+3,\dots,6u+e(s-2)+6k+2\}$. If $i$ is odd, we have $c_i=c_{i+e}=1$ and $(a_i,a_{i+1},a_{i+2})\in \{(0,0,1),(1,1,1),(1,0,0)\}$ and if $i$ is even $c_i=c_{i+e}=0$ and $(a_i,a_{i+1},a_{i+2})\in \{(0,1,1),(1,1,0),(0,0,0)\}$. \\
	\underline{Case a.5:} $i \in \{6u+e(s-2)+6k+3,\dots,6u+e(s-2)+6k+5\}$. If $i=6u+e(s-2)+6k+3$ then $c_{i}=0$, $c_{i+e}=c_{n-1}=0$ and $(a_i,a_{i+1},a_{i+2})=(0,1,1)$. For $i=6u+e(s-2)+6k+4$ we have $c_{i}=0$, $c_{i+e}=c_{0}=0$ and $(a_i,a_{i+1},a_{i+2})=(1,1,0)$ and for $i=6u+e(s-2)+6k+5$ we have $c_{i}=1$, $c_{i+e}=c_{1}=1$ and $(a_i,a_{i+1},a_{i+2})=(1,0,0)$. \\
	\underline{Case a.6:} $i$ is in the second block of $x_1$'s, i.e. $i \in \{6u+e(s-2)+6k+6,\dots,12u+e(s-2)+6k+5\}$. If $i$ is even, we have $c_{i}=c_{i+e}=1$ and $(a_i,a_{i+1},a_{i+2})\in \{(0,1,1),(1,1,0),(0,0,0)\}$. If $i$ is odd and $i \neq 12u+e(s-2)+6k+5$ we have $c_{i}=c_{i+e}=0$ and $(a_i,a_{i+1},a_{i+2})\in \{(0,0,1),(1,1,1),(1,0,0)\}$. If $i=12u+e(s-2)+6k+5$ then $c_i=1$, $c_{i+e}=0$ and $(a_i,a_{i+1},a_{i+2})=(1,0,1)$. \\
	\underline{Case a.7:} $i$ is in the subsequence $z$, i.e. $i \in \{12u+e(s-2)+6k+5,6u+e(s-1)+6k+1\}$. If $i$ is even then $c_i=0$, $c_{i+e}=1$ and $(a_i,a_{i+1},a_{i+2})=(0,1,0)$. If $i$ is odd then $c_i=1$, $c_{i+e}=0$ and $(a_i,a_{i+1},a_{i+2})=(1,0,1)$.\\
	So Eq.\eqref{eq:gcd1_2} holds for all $i \neq 0$. \\
	We state the $r$-ordered carry sequences for the other cases:\\
	\emph{Case (b)}: 
	\begin{equation*}
		c=(-1,1,s_1,\underbrace{s_2,\dots,s_2}_{s\text{-times}}).
	\end{equation*}
		\emph{Case (c)}: 
	\begin{equation*}
		c=(0,0,1,s_1,\underbrace{s_2,\dots,s_2}_{s\text{-times}},0).
	\end{equation*}
		\emph{Case (d)}: 
	\begin{equation*}
		c=(0,0,1,0,1,s_1,\underbrace{s_2,\dots,s_2}_{s\text{-times}}).
	\end{equation*}
	\qed
\end{proof}

Note that Proposition \ref{prop:gcd1_2} lists all possible options. Indeed, the cases $t=6u$ and $t=6u+3$ do not occur because in these cases $n=se+t$ is divisible by $3$, so $e=6k+3$ is never invertible modulo $n$. 

\begin{corollary} \label{cor:algdeg}
	Let $n \in \N$ and $K_r$ be the $r$-th Kasami exponent with $\gcd(n,r)=1$. Let $K_r^{-1}$ be the inverse of $K_r$ modulo $2^n-1$. Then $\wt(K_r^{-1})=\frac{n+1}{2}$ for $n \equiv 0 \pmod 3$. Moreover, we have 
	\begin{equation*}
	\wt(K_r^{-1})\geq \begin{cases} \frac{n+2}{3} & \text{ if }n\equiv 1 \pmod 3 \\
													\frac{n+1}{3} & \text{ if }n\equiv 2 \pmod 3.
									\end{cases}
	\end{equation*}

	The lower bound is attained if and only if $e=3$.
\end{corollary}
\begin{proof}
If $n \equiv 0 \pmod 3$ then $e$ is not divisible by $3$ since $\gcd(e,n)=1$. The result then follows from Proposition \ref{prop:gcd1_1}.

	For the other cases, using the notation of Proposition \ref{prop:gcd1_2}, the binary weight $\wt(K_r^{-1})$ is minimal when $s$ is maximal. For $n=se+t$ with $0 <t<e$ this clearly implies minimizing $e$, so $e=3$ and $t \in \{1,2\}$. For these cases we have 
	\begin{equation*}
		\wt(K_r^{-1})=\begin{cases} \frac{n-\frac{n-1}{3}+1}{2}=\frac{n+2}{3} & \text{ if }t=1 \\
													\frac{n-\frac{n-2}{3}}{2}=\frac{n+1}{3} & \text{ if }t=2
									\end{cases}
	\end{equation*}
	and the result follows.
	\qed
\end{proof}
It is known that a vectorial Boolean function $f$ is always CCZ-equivalent to its inverse $f^{-1}$. It is however not clear when a function is EA-equivalent  to its inverse. Since EA equivalence preserves the algebraic degree, we get the following easy corollary.
\begin{corollary}
Let $n \in \N$ odd and $K_r$ be the $r$-th Kasami exponent with $\gcd(n,r)=1$ and $r<\frac{n}{2}$. Let $f=x^{K_r}$ be the $r$-th Kasami function on $\F_{2^n}$. If $n \equiv 0 \pmod 3$ and $r \neq \frac{n-1}{2}$ then $f$ is not EA equivalent to $f^{-1}$. If $n \not\equiv 0 \pmod 3$ and $r<\frac{n-2}{3}$ then $f$ is not EA-equivalent to $f^{-1}$.
\end{corollary}

\subsection{The case $\frac{n}{\gcd(n,r)}$ odd}

We now deal with the Kasami exponents $K_r$ with $\gcd(n,r)>1$ and $\frac{n}{\gcd(n,r}$ odd. While these Kasami exponents are not APN, they still have some interesting properties. For example, for $\gcd(r,n)=2$ and $\frac{n}{2}$ odd, the function $x \mapsto x^{K_r}$ (and thus also its inverse) is a permutation with differential uniformity 4 (see Table~\ref{t:4uni}). 

Since $\gcd(n,r)>1$, we cannot use the $r$-ordering of sequences that we used in the previous section. Just like in the case of Gold functions, we will thus use $r$-matrices (introduced in Definition~\ref{def:rmatrix}).

In accordance to the notation used in the previous subsection, we denote by $e$ the least positive residue of the inverse of $\frac{r}{\gcd(n,r)}$ modulo $\frac{n}{\gcd(n,r)}$. 
Since $\gcd(n,r)=\gcd(n-r,r)$, $\frac{n}{\gcd(n,r)}$ odd and $K_r$ is cyclotomic equivalent to $K_{n-r}$, it again suffices to determine the inverses of $K_r$ where $e$ is odd. 

Using $r$-matrices, Theorem \ref{thm:xiang} takes on the following form.

\begin{theorem} \label{thm:gcd>1}
	Let $a \in \{1,\dots,2^n-2\}$, $n \in N$ and $K_r$ be the $r$-th Kasami exponent with $\gcd(r,n)=d$ and $e$ be the least positive residue of $\frac{r}{d}$ modulo $\frac{n}{d}$. Moreover, let \begin{equation*}
		M_{a,r}=\begin{pmatrix}
			a_{0,0}&a_{0,1}&a_{0,2}&\hdots&a_{0,\frac{n}{d}-1} \\
			a_{1,0}&a_{1,1}&a_{1,2}&\hdots&a_{1,\frac{n}{d}-1} \\
			\vdots& & & &\vdots \\
			a_{d-1,0}&a_{d-1,1} &a_{d-1,2}& \hdots & a_{d-1,\frac{n}{d}-1}
		\end{pmatrix}
	\end{equation*} be the $r$-matrix of $a$, i.e. $a\equiv \sum_{i=0}^{d-1}\sum_{j=0}^{\frac{n}{d}-1}a_{i,j}2^{i-jr} \pmod{2^n-1}$. 	The following are equivalent:
	\begin{enumerate}[label=(\alph*)]
		\item $a$ is the inverse of $K_r$ modulo $2^n-1$.
			\item There exists an $r$-matrix for the carry sequence $c$ of the form 
			\begin{equation*}
		M_{c,r}=\begin{pmatrix}
			c_{0,0}&c_{0,1}&c_{0,2}&\hdots&c_{0,\frac{n}{d}-1} \\
			c_{1,0}&c_{1,1}&c_{1,2}&\hdots&c_{1,\frac{n}{d}-1} \\
			\vdots& & & &\vdots \\
			c_{d-1,0}&c_{d-1,1} &c_{d-1,2}& \hdots & c_{d-1,\frac{n}{d}-1}
		\end{pmatrix}
	\end{equation*} 
			with $c_{i,j} \in \{-1,0,1\}$  such that the following equations hold:
	\begin{align} 
		2c_{0,0}-c_{d-1,e}+1&=a_{0,2}-a_{0,1}+a_{0,0} \label{eq:gcds_1} \\
		2c_{0,j}-c_{d-1,j+e}&=a_{0,j+2}-a_{0,j+1}+a_{0,j}  \text{ for all }j\in \{1,\dots,\frac{n}{d}-1\} \label{eq:gcds_2}\\
		2c_{i,j}-c_{i-1,j}&=a_{i,j+2}-a_{i,j+1}+a_{i,j} \text{ for all }i \in \{1,\dots,d-1\}, j \in \{0,\dots,\frac{n}{d}-1\}. \label{eq:gcds_3}
	\end{align}
	\end{enumerate}
	The carry sequence (and thus its associated $r$-matrix) in (b) is unique.
\end{theorem}
\begin{proof}
	The Theorem follows immediately from Theorem \ref{thm:xiang} and the definition of the $r$-matrix. The process is identical to the corresponding case for the Gold function in Theorem~\ref{thm:gcd>1_gold}.
	\qed
\end{proof}

Again, we find $M_{a,r}$ and $M_{c,r}$ such that Eq.~\eqref{eq:gcds_1}-\eqref{eq:gcds_3} hold. These verifications become quite tedious (especially since we have to distinguish several cases). However, the basic idea does not change: The $r$-matrices of the carry sequences have a visible structure that can be used to determine the inverse. It turns out that the inverse of $K_r$ on modulo $2^n-1$ with $\gcd(r,n)=d$ is closely related to the inverse of $K_{\frac{r}{d}}$ modulo $2^{\frac{n}{d}}-1$ which was already determined in the previous section.  To improve readability, we first deal with the case $\frac{n}{\gcd(n,r)}=6v+3$ for a $v \in \N_0$ separately.

\begin{proposition}\label{prop:gcds1}
	Let $n\in \N$ and $K_r$ be the $r$-th Kasami exponent with $\gcd(r,n)=d>1$ and $\frac{n}{d}=6v+3$. Let $e$ be the least positive residue of the inverse of $\frac{r}{d}$ modulo $\frac{n}{d}$. Then $K_r^{-1}\equiv \sum_{i=0}^{d-1}\sum_{j=0}^{\frac{n}{d}-1}a_{i,j}2^{i-jr} \pmod{2^n-1}$ where the values $a_{i,j}$ are the entries of the matrix $M_{a,r}$ 
			\begin{equation*}M_{a,r}=
		\begin{pmatrix}
			a_1 \\
							\vdots \\
			a_1 \\
			a_2
		\end{pmatrix},
			\end{equation*}
		 where the rows $a_1$ and $a_2$ are defined as follows:
	\begin{enumerate}[label=(\alph*)]
		\item If $e=6k+1$:
		\begin{align*}
			a_1&=(0,0,\underbrace{x_1,\dots,x_1}_{\frac{n/d-e-2}{6}\text{ -times}}, \underbrace{x_2\dots,x_2}_{k\text{ -times}},0) \\
			a_2 &= (1,x_3,\underbrace{x_4\dots,x_4}_{k\text{ -times}}).
		\end{align*}
		
		\item If $e=6k+5$:		
	\begin{align*}
			a_1&=(0,1,0,0,0,\underbrace{x_4,\dots,x_4}_{\frac{n/d-e-4}{6}\text{ -times}}, 1,1,0,0,\underbrace{x_6\dots,x_6}_{k\text{ -times}},0) \\
			a_2 &= (0,x_5,\underbrace{x_2\dots,x_2}_{k\text{ -times}}),
		\end{align*}
		\end{enumerate}
		where $x_1=(0,0,1,1,1,0)$, $x_2=(0,0,0,1,1,1)$, $x_4=(1,1,1,0,0,0)$, $x_6=(0,1,1,1,0,0)$ are sequences of length $6$, $x_3=(1,0,1,0\dots,1,0,1,0)$ is a sequence of length $\frac{n}{d}-e$ and $x_5=(0,1,0,1,\dots,0,1)$ is a sequence of length $\frac{n}{d}-e$.

	In both cases we have $\wt(K_r^{-1})=\frac{n-3d+4}{2}$.
\end{proposition}
\begin{proof}
 \emph{Case (a):}
	The $r$-matrix of the corresponding carry sequence is
		\begin{equation*}M_{c,r}=(c_{i,j})=
		\begin{pmatrix}
			c'  \\
							\vdots  \\
							c' \\
							c'' 
		\end{pmatrix}
	\end{equation*}	
	where $c''=(c_0,\dots,c_{\frac{n}{d}-1})=(0,1,0,1,\dots,0,1,0,1,0)$ and $c'=(c_e-1,c_{e+1},\dots,c_{\frac{n}{d}-1},c_0,c_1,\dots,c_{e-1})$. 
	
Using Theorem \ref{thm:gcd>1}, we just have to verify Eq.~\eqref{eq:gcds_1}-~\eqref{eq:gcds_3}. For our choice of $M_{c,r}$, we have in Eq.~\eqref{eq:gcds_1} $2c_{0,0} - c_{d-1,e} + 1 = 2(c_e-1)-c_e+1 = c_e-1$. Similarly, in Eq.~\eqref{eq:gcds_3} we have for $0<i<d-1$ and $j=0$ the relation $2c_{i,j}-c_{i-1,j} = 2(c_e-1)-(c_e-1)=c_e-1$. From these two observations, we conclude
\begin{equation}
		c_{e}-1=a_{i,2}-a_{i,1}+a_{i,0} \text{ for all }i \in \{0,\dots,d-2\} \label{eq:gcds_1_p1}.
\end{equation}
For $j \neq 0$ we have for $i=0$ (consulting Eq.~\eqref{eq:gcds_2}) $2c_{0,j}-c_{d-1,e+j} = 2(c_{e+j})-c_{e+j} = c_{e+j}$. Looking at Eq.~\eqref{eq:gcds_3} for $j\neq 0$ and $i \neq 0$, we have $2c_{i,j}-c_{i-1,j} = 2c_{e+j}-c_{e+j} =c_{e+j}$. We conclude
\begin{equation}
		c_{e+j}=a_{i,j+2}-a_{i,j+1}+a_{i,j} \text{ for all }i \in \{0,\dots,d-1\}, j \in \{1,\dots,\frac{n}{d}-1\} \label{eq:gcds_2_p1}.
\end{equation}
Let us now consider the case $i = d-1$, $j\neq 0$. Then Eq.~\eqref{eq:gcds_3} becomes $2c_{d-1,j}-c_{d-2,j} = 2c_j-c_{e+j}$ and we get
\begin{equation}
		2c_{j}-c_{e+j} =a_{d-1,j+2}-a_{d-1,j+1}+a_{d-1,j} \text{ for all }j \in \{1,\dots,\frac{n}{d}-1\}\label{eq:gcds_3_p1}.
\end{equation}
Finally, for the case $i = d-1$ and $j = 0$ we have (again considering Eq.~\eqref{eq:gcds_3}) $2c_{d-1,0}-c_{d-2,0} = 2c_0-(c_e-1)=2c_0-c_e+1$. We conclude
\begin{equation}
		2c_{0}-c_{e}+1 =a_{d-1,2}-a_{d-1,1}+a_{d-1,0} \label{eq:gcds_4_p1}
\end{equation}

Observe that, by Proposition \ref{prop:gcd1_1}, $a_2$ is the $r$-ordered sequence of the inverse of $K_{\frac{r}{d}}$ modulo $2^{\frac{n}{d}}-1$ with the corresponding carry sequence $c''$. Theorem \ref{thm:gcd1} then shows that Eq.~\eqref{eq:gcds_3_p1} and~\eqref{eq:gcds_4_p1} are satisfied. We check Eq.~\eqref{eq:gcds_1_p1} and~\eqref{eq:gcds_2_p1} by hand. In both equations we do not consider the last row of $M_{a,r}$ and since all but the last row in $M_{a,r}$ are identical, it suffices to check the first row.

Eq.~\eqref{eq:gcds_1_p1} holds because $c_e=1$ and $a_{0,2}=a_{0,1}=a_{0,0}=0$. We check Eq.~\eqref{eq:gcds_2_p1}: If $e+j<n$ and $j$ odd, then $c_{e+j}=0$ and $(a_{0,j},a_{0,j+1},a_{0,j+2})\in \{(0,0,0),(1,1,0),(0,1,1)\}$. If $e+j<n$ and $j$ even, then $c_{e+j}=1$ and $(a_{0,j},a_{0,j+1},a_{0,j+2})\in \{(1,0,0),(1,1,1),(0,0,1)\}$. If $e+j \geq n$ and $j$ is odd then $c_{e+j}=1$ and $(a_{0,j},a_{0,j+1},a_{0,j+2})\in \{(1,0,0),(1,1,1),(0,0,1)\}$ and if $e+j \geq n$ and $j$ is even then $c_{e+j}=0$ and $(a_{0,j},a_{0,j+1},a_{0,j+2})\in \{(0,0,0),(1,1,0),(0,1,1)\}$. \\

 \emph{Case (b):} 
The proof is similar to the proof of the first case. We define the $r$-matrix of the corresponding carry sequence
		\begin{equation*}M_{c,r}=(c_{i,j})=
		\begin{pmatrix}
			c'  \\
							\vdots  \\
							c'  \\
							c'' 
		\end{pmatrix}
	\end{equation*}	
where $c''=(c_0,\dots,c_{\frac{n}{d}-1})=(0,0,1,0,1,0,1,\dots,0,1,0,1)$ and $c'=(c_e-1,c_{e+1},\dots,c_{\frac{n}{d}-1},c_0,c_1,\dots,c_{e-1})$. This leads to precisely the same equations~\eqref{eq:gcds_1_p1}-\eqref{eq:gcds_4_p1}. Again, by Proposition \ref{prop:gcd1_1}, $a_2$ and $c''$ are the $r$-ordered sequences of the inverse of the Kasami exponent $K_{\frac{r}{d}}$ modulo $2^{\frac{n}{d}}-1$ and the corresponding carry sequence, respectively. The validity of Eq.~\eqref{eq:gcds_3_p1} and~\eqref{eq:gcds_4_p1} follows. Equations~\eqref{eq:gcds_1_p1} and~\eqref{eq:gcds_2_p1} can be checked just as in the previous case; we omit the calculations.

By adding all entries in $M_{c,r}$, we see that in both cases the weight of the carry sequence is $d\frac{n/d-3}{2}+1=\frac{n-3d+2}{2}$. Lemma \ref{lem:carry} then implies $\wt(K_r^{-1})=\frac{n-3d+4}{2}$. 
\qed
\end{proof}
Note that the case $e=6k+3$ does not occur because $e$ is invertible modulo $\frac{n}{d}=6v+3$. We now deal with the remaining cases $\frac{n}{d}=6v+1$ and $\frac{n}{d}=6v+5$. 

\begin{proposition} \label{prop:gcds2}
	Let $n\in \N$ and $K_r$ be the $r$-th Kasami exponent with $\gcd(r,n)=d$ and $\frac{n}{d}$ odd. Let $e$ be the least positive residue of the inverse of $\frac{r}{d}$ modulo $\frac{n}{d}$ and $\frac{n}{d}=se+t$, $0 \leq t <e$. Then $K_r^{-1}\equiv \sum_{i=0}^{d-1}\sum_{j=0}^{\frac{n}{d}-1}a_{i,j}2^{i-jr} \pmod{2^n-1}$ where the values $a_{i,j}$ are the entries of the matrix
	\begin{equation*}M_{a,r}=
		\begin{pmatrix}
			a_1 \\
			a_2 \\
			\vdots \\
			a_2
		\end{pmatrix}.
	\end{equation*}
	Here, $a_1$ is the sequence of the inverse of $K_{\frac{r}{d}}$ modulo $2^{\frac{n}{d}}-1$ in $r$-ordering as determined in the previous section and $a_2$ is as follows. We use the auxiliary sequences $x_1=(0,0,0,1,1,1)$, $x_2=(1,1,0,0,0,1)$, $x_3=(0,1,1,1,0,0)$ of length $6$ and 
	\begin{align*}
	y&=(0,0,0,\underbrace{x_3,\dots,x_3}_{k \text{-times}},0,1,1,\underbrace{x_1,\dots,x_1}_{k \text{-times}}) \\
	z&=(0,1,1,\underbrace{x_1,\dots,x_1}_{k \text{-times}},0,0,0,\underbrace{x_3,\dots,x_3}_{k \text{-times}})
	\end{align*}
	of length $12k+6$.
	\begin{enumerate}[label=(\alph*)]
		\item If $e=6k+1$ and $\frac{n}{d}=6v+1$
		\begin{equation*}
			a_2=(\underbrace{x_1,\dots,x_1}_{v \text{-times}},0).
		\end{equation*}
		\item If $e=6k+1$ and $\frac{n}{d}=6v+5$
		\begin{equation*}
			a_2=(\underbrace{x_2,\dots,x_2}_{v \text{-times}},1,1,0,0,0).
		\end{equation*}
		\item If $e=6k+5$ and $\frac{n}{d}=6v+1$
		\begin{equation*}
			a_2=(0,\underbrace{x_1,\dots,x_1}_{v \text{-times}}).
		\end{equation*}
		\item If $e=6k+5$ and $\frac{n}{d}=6v+5$
		\begin{equation*}
			a_2=(0,1,1,\underbrace{x_1,\dots,x_1}_{v \text{-times}},0,0).
		\end{equation*}
		\item If $e=6k+3$ and $t=6u+1$
		\begin{equation*}
			a_2=(\underbrace{x_1,\dots,x_1}_{u \text{-times}},\underbrace{y,\dots,y}_{\frac{s}{2} \text{-times}},0).
		\end{equation*}
		\item If $e=6k+3$ and $t=6u+2$
		\begin{equation*}
			a_2=(0,1,\underbrace{x_1,\dots,x_1}_{u \text{-times}},\underbrace{y,\dots,y}_{\frac{s-1}{2} \text{-times}},0,0,0,\underbrace{x_3,\dots,x_3}_{k \text{-times}}).
		\end{equation*}
		\item If $e=6k+3$ and $t=6u+4$
		\begin{equation*}
			a_2=(0,0,0,\underbrace{x_3,\dots,x_3}_{u \text{-times}},\underbrace{z,\dots,z}_{\frac{s-1}{2} \text{-times}},0,1,1,\underbrace{x_1,\dots,x_1}_{u \text{-times}},0).
		\end{equation*}
		\item If $e=6k+3$ and $t=6u+5$
		\begin{equation*}
			a_2=(0,1,1,0,0,\underbrace{x_3,\dots,x_3}_{u \text{-times}},\underbrace{z,\dots,z}_{\frac{s}{2} \text{-times}}).
		\end{equation*}
	\end{enumerate}
	In the cases (a)-(d) we have $\wt(K_r^{-1})=\frac{n-d+2}{2}$, in the cases (e) and (h) $\wt(K_r^{-1})=\frac{n-d(s+1)+2}{2}$ and in cases (f) and (g) $\wt(K_r^{-1})=\frac{n-d(s+2)+2}{2}$.
\end{proposition}

\begin{proof}
	In all cases the $r$-matrix of the carry sequence $c$ has identical rows, i.e.
	\begin{equation*}
		M_{c,r}=\begin{pmatrix}
			c' \\
			\vdots \\
			c'
		\end{pmatrix},
	\end{equation*}
	where $c'=(c_0,\dots,c_{\frac{n}{d}-1})$ is the $r$-ordered carry sequence for the inverse of $K_{\frac{r}{d}}$ modulo $2^{\frac{n}{d}}-1$ determined in the proofs of Propositions \ref{prop:gcd1_1} and \ref{prop:gcd1_2}. With this carry sequence, the equations~\eqref{eq:gcds_1}-\eqref{eq:gcds_3} of Theorem \ref{thm:gcd>1} take on the following form:
	\begin{align} 
		2c_{0}-c_{e}+1&=a_{0,2}-a_{0,1}+a_{0,0} \label{eq:gcds_1_p2} \\
		2c_{j}-c_{j+e}&=a_{0,j+2}-a_{0,j+1}+a_{0,j}  \text{ for all }j\in \{1,\dots,\frac{n}{d}-1\} \label{eq:gcds_2_p2}\\
		c_{j}&=a_{i,j+2}-a_{i,j+1}+a_{i,j} \text{ for all }i \in \{1,\dots,d-1\}, j \in \{0,\dots,\frac{n}{d}-1\}. \label{eq:gcds_3_p2}
	\end{align}
The validity of Eq.~\eqref{eq:gcds_1_p2} and~\eqref{eq:gcds_2_p2} follows from Theorem \ref{thm:gcd1} and the choice of $a_1$ and $c'$. So we only need to verify Eq.~\eqref{eq:gcds_3_p2} for each case. We will show the verification for the first case, the other cases are identical in nature.

In Case (a) we have $c'=(0,1,0,1,\dots,0,1,0)$ from Proposition \ref{prop:gcd1_1}, i.e. $c_j$ is $0$ if $j$ is even and $1$ of $j$ is odd. When $j$ is odd, then $(a_{i,j},a_{i,j+1},a_{i,j+2})\in \{(0,0,1),(1,1,1),(1,0,0)\}$ and if $j$ is even then  $(a_{i,j},a_{i,j+1},a_{i,j+2})\in \{(0,0,0),(1,1,0),(0,1,1)\}$ for all $i>0$, so Eq.~\eqref{eq:gcds_3_p2} holds.

Using Lemma \ref{lem:carry}, we have 
\begin{equation*}
	\wt(K_r^{-1})=\wt(c)+1=d\wt(c')+1=d(\wt(K_{\frac{r}{d}}^{-1})-1)+1,
\end{equation*}
 where $K_{\frac{r}{d}}^{-1}$ is the least positive residue of the inverse of $K_{\frac{r}{d}}$ modulo $2^{\frac{n}{d}}-1$. The results on the binary weights then follow from the results in Propositions~\ref{prop:gcd1_1} and \ref{prop:gcd1_2}. For example for the cases (a)-(d), Proposition~\ref{prop:gcd1_1} yields $\wt(K_{\frac{r}{d}}^{-1})=\frac{\frac{n}{d}+1}{2}$. This leads to $\wt(K_r^{-1})=d\frac{\frac{n}{d}-1}{2}+1=\frac{n-d+2}{2}$.
\qed
\end{proof}

Propositions~\ref{prop:gcds1} and \ref{prop:gcds2} show that $K_r^{-1}$ has a strong structure because its $r$-matrix has $d-1$ identical rows. By the definition of the $r$-matrix, this means that $K_r^{-1}$ has $\frac{n}{d}$ runs of $(d-1)$ consecutive ones or zeroes.

The results presented in this section yield the following result for the binary weight of the inverse of Kasami exponents.

 \begin{corollary} \label{cor:algdeg2}
	Let $n \in \N$ and $K_r$ be the $r$-th Kasami exponent with $\gcd(n,r)=d$ and $\frac{n}{d}$ odd. Let $K_r^{-1}$ be the inverse of $K_r$ modulo $2^n-1$. Then $\wt(K_r^{-1})=\frac{n-3d+4}{2}$ for $n \equiv 0 \pmod 3$ and $\wt(K_r^{-1})\leq \frac{n-d+2}{2}$ for $n \not\equiv 0 \pmod 3$. Moreover, we have 
	\begin{equation*}
	\wt(K_r^{-1})\geq \begin{cases} \frac{n-d+3}{3} & \text{ if }\frac{n}{d}\equiv 1 \pmod 3 \\
													\frac{n-2d+3}{3} & \text{ if }\frac{n}{d}\equiv 2 \pmod 3.
									\end{cases}
	\end{equation*}
\end{corollary}
\begin{proof}
For $n \equiv 0 \pmod 3$ the result follows from Proposition \ref{prop:gcds1}.

	For the other cases, using the notation of Proposition \ref{prop:gcds2}, the binary weight $\wt(K_r^{-1})$ is minimal when $e$ is divisible by $3$ and $s$ is maximal. For $n/d=se+t$ with $0 <t<e$ this clearly implies minimizing $e$, so $e=3$ and $t \in \{1,2\}$. With Case (e) and (f) from Proposition \ref{prop:gcds2}, we have
	\begin{equation*}
		\wt(K_r^{-1})\geq\begin{cases} \frac{1}{2}(n-\frac{n+2d}{3}+2)=\frac{n-d+3}{3} & \text{ if }t=1 \\
													\frac{1}{2}(n-\frac{n+4d}{3}+2)=\frac{n-2d+3}{3} & \text{ if }t=2
									\end{cases}
	\end{equation*}
	and the result follows.
	\qed
\end{proof}

\subsection{The case $\frac{n}{\gcd(n,r)}$ even}

We now deal with the case $\frac{n}{\gcd(n,r)}$ even. Proposition~\ref{prop:invexist} implies that if $K_r$ is invertible modulo $2^n-1$ then both $n$ and $r$ are even and $\frac{n}{\gcd(n,r)}$ is not divisible by $3$. We will again denote by $e$ the inverse of $\frac{r}{\gcd(n,r)}$ modulo $\frac{n}{\gcd(n,r)}$. Note that since  $\frac{n}{\gcd(n,r)}$ is even, $e$ must be odd. 

	\begin{proposition} \label{prop:gcds3}
	Let $n\in \N$ and $K_r$ be the $r$-th Kasami exponent with $\gcd(r,n)=d$, $r$ even, $\frac{n}{d}$ even and not divisible by $3$. Then $K_r^{-1}\equiv \sum_{i=0}^{d-1}\sum_{j=0}^{\frac{n}{d}-1}a_{i,j}2^{i-jr} \pmod{2^n-1}$ where the values $a_{i,j}$ are the entries of the matrix
	\begin{equation*}M_{a,r}=
		\begin{pmatrix}
			a_1 \\
			x \\
			y \\
			\vdots \\
			x\\
			y\\
			x
		\end{pmatrix}.
	\end{equation*}
	where $a_1,x,y$ are as follows. We use the auxiliary sequences $x_1=(1,1,0,0,0,1)$ and $x_2=(1,0,0,0,1,1)$ of length 6.
	\begin{enumerate}[label=(\alph*)]
		\item If $\frac{n}{d}=6k+2$ then
		\begin{equation*}
			a_1=(1,1,\underbrace{x_1,\dots,x_1}_{k \text{ -times}}), x=(1,0,1,0,\dots,1,0), y=(0,1,0,1,\dots,0,1).
		\end{equation*}
		\item If $\frac{n}{d}=6k+4$ then
		\begin{equation*}
			a_1=(1,0,1,1,\underbrace{x_2,\dots,x_2}_{k \text{ -times}}), x=(0,1,0,1,\dots,0,1), y=(1,0,1,0,\dots,1,0).
		\end{equation*}
	\end{enumerate}
	In both cases we have $\wt(K_r^{-1})=\frac{n+2}{2}$.
\end{proposition}

\begin{proof}
	\emph{Case (a)}: The $r$-matrix of the carry sequence is 
	\begin{equation*}M_{c,r}=(c_{i,j})=
		\begin{pmatrix}
			 0&1&0&1&\hdots&0&1\\
			 1&0&1&0&\hdots&1&0\\
			\vdots& & & & & &  \vdots \\
			 0&1&0&1&\hdots&0&1\\
			 1&0&1&0&\hdots&1&0
		\end{pmatrix}.
	\end{equation*}
	We check Eq.~\eqref{eq:gcds_1}-\eqref{eq:gcds_3} from Theorem \ref{thm:gcd>1}. 
	
	Eq.~\eqref{eq:gcds_1} holds because $c_{0,0}=0$, $c_{d-1,e}=0$ (recall that $e$ is odd) and $a_{0,2}=a_{0,1}=a_{0,0}=1$.
	
	We verify Eq.~\eqref{eq:gcds_2}: If $j$ is odd then $c_{0,j}=c_{d-1,j+e}=1$ and $(a_{0,j},a_{0,j+1},a_{0,j+2}) \in \{(1,0,0),(1,1,1),(0,0,1)\}$. If $j>0$ is even, then $c_{0,j}=c_{d-1,j+e}=0$ and $(a_{0,j},a_{0,j+1},a_{0,j+2}) \in \{(1,1,0),(0,1,1),(0,0,0)\}$.
	
	Lastly, we verify Eq.~\eqref{eq:gcds_3}: If $i+j$ is even then $c_{i,j}=0$, $c_{i-1,j}=1$, $a_{i,j+2}=a_{i,j}=0$ and $a_{i,j+1}=1$. If $i+j$ is odd then $c_{i,j}=1$, $c_{i-1,j}=0$, $a_{i,j+2}=a_{i,j}=1$ and $a_{i,j+1}=0$.\\
	
	\emph{Case (b)}: In this case, the $r$-matrix of the carry sequence is 
		\begin{equation*}M_{c,r}=(c_{i,j})=
		\begin{pmatrix}
			 1&0&1&0&\hdots&1&0\\
			 0&1&0&1&\hdots&0&1\\
			\vdots& & & & & &  \vdots \\
			 1&0&1&0&\hdots&1&0\\
			 0&1&0&1&\hdots&0&1
		\end{pmatrix}.
	\end{equation*}
	Eq.~\eqref{eq:gcds_1} is valid since $c_{0,0}=1$, $c_{d-1,e}=1$ and $a_{0,0}=a_{0,2}=1$ and $a_{0,1}=0$. The verification process for Eq.~\eqref{eq:gcds_2} and 
	\eqref{eq:gcds_3} is identical to Case (a) with odd and even swapped.
	\qed
\end{proof}

\subsection{Kasami inverses with special structure}
We now investigate cases where the inverses of Kasami exponents have some special structure. These cases will also illustrate the results in the previous sections and show how to get from the representation using $r$-matrices to the ``usual'' binary representation.

 In  \cite[Proposition 3.13]{FFA:KyuSud14}, it was shown that the inverse of $K_r$ modulo $2^{5r}-1$ is cyclotomic equivalent to the Kasami exponent $K_{2r}$. It was conjectured that $K_r^{-1}$ modulo $ 2^{\frac{5r}{b}}-1$ for $b|r$ and $5 \nmid b$ is always cyclotomic equivalent to a Kasami exponent. This conjecture can be proven using Proposition~\ref{prop:gcds2}. 

\begin{proposition} \label{prop:conjecture}
	Let $d=\frac{r}{b}$ with $b|r$, $n=5d$ and $K_r^{-1}$ be the least positive residue of the inverse of $K_r$ modulo $2^n-1$. Then 
	\begin{equation*}
		K_r^{-1} \equiv \begin{cases}
				2^{2d}K_{2d} \pmod {2^n-1} & \text{if }b\equiv 1 \pmod 5 \\
				2^{2d}K_{d} \pmod {2^n-1} & \text{if }b\equiv 2 \pmod 5 \\
				2^{2(d-r)}K_{d} \pmod {2^n-1} & \text{if }b\equiv 3 \pmod 5\\
				2^{2(d-r)}K_{2d} \pmod {2^n-1} & \text{if }b\equiv 4 \pmod 5.
		\end{cases}
	\end{equation*}

\end{proposition}
\begin{proof}
	We use the notation of Proposition~\ref{prop:gcds2}. We have $d=\gcd(n,r)=\frac{r}{b}$ and $\frac{n}{d}=5$. Further, we have $\frac{r}{d}=b$. The only two possible odd values for $e$ are $e=1$ and $e=3$ that are attained for $b\equiv 1 \pmod 5$ and $b\equiv 2 \pmod 5$, respectively. These correspond to case (b) and (f) in Proposition~\ref{prop:gcds2}. We get $K_r^{-1}\equiv \sum_{i=0}^{d-1}\sum_{j=0}^{4}a_{i,j}2^{i-jr} \pmod{2^n-1}$ where the values $a_{i,j}$ are the entries of the matrix $M_1$ if $e=1$ and $M_2$ if $e=3$:
	\begin{equation*}
		M_1=\begin{pmatrix}
			1&1&0&1&0 \\
			1&1&0&0&0 \\
			\vdots& & & & \vdots  \\
			1&1&0&0&0
		\end{pmatrix}, 
		M_2=\begin{pmatrix}
			0&1&0&0&1 \\
			0&1&0&0&0 \\
			\vdots& & & & \vdots \\
			0&1&0&0&0
		\end{pmatrix}.
	\end{equation*}
	We now write $K_r^{-1}$ in its usual binary representation. To do this, we write from right to left in the following way: We start with the first column, and then proceed in steps of length $e$ to the left (cyclically). So, for the case $e=1$, we start with column $0$ of $M_1$, then column $4$, then $3$, then $2$ and then $1$, resulting in:
	\begin{equation*}
	K_r^{-1}=(\underbrace{1,1,\dots,1,1}_{d\text{-times}},\underbrace{0,0,\dots,0,0}_{2d-1 \text{-times}},1,\underbrace{0,0,\dots,0,0}_{d \text{-times}},\underbrace{1,1,\dots,1,1}_{d \text{-times}})
	\end{equation*}
	and for the case $e=3$ the order of the columns is $0,2,4,1,3$, resulting in:
	\begin{equation*}
	K_r^{-1}=(\underbrace{0,0,\dots,0,0}_{d \text{-times}},\underbrace{1,1,\dots,1,1}_{d\text{-times}},\underbrace{0,0,\dots,0,0}_{d-1 \text{-times}},1,\underbrace{0,0,\dots,0,0}_{2d \text{-times}})
	\end{equation*}
	In the first case, we have $K_r^{-1}\equiv 2^{2d}K_{2d} \pmod {2^n-1}$ and in the second case $K_r^{-1}\equiv2^{2d}K_{d}\pmod {2^n-1}$.
	If $e=2$ and $e=4$  (corresponding to the values $b\equiv 3 \pmod 5$ and $b\equiv 4 \pmod 5$) we use the relation $K_{r}^{-1} \equiv 2^{-2r}K_{n-r}^{-1}\pmod{2^n-1}$ and apply the procedure above to $K_{n-r}$.
	\qed
\end{proof}

In fact, in  \cite{FFA:KyuSud14} several nice formulas for the inverses of $K_r$ modulo $2^{kr}-1$ for small fixed values of $k$ have been found. Our framework gives an explanation why these inverses have a strong structure: We have $\frac{kr}{\gcd(r,kr)}=k$, so the $r$-matrices always have $k$ columns. By Proposition~\ref{prop:gcds1} and \ref{prop:gcds2}, all but one row in the $r$-matrix are identical, so we get long runs of zeroes and ones (as observed in the proof of Proposition~\ref{prop:conjecture}). All of these formulas can also be obtained using our framework. In particular, it was shown in \cite{FFA:KyuSud14} that if $n=\frac{3r}{b}$ with $b|r$ and $\gcd(3,b)=1$ then the inverse of $K_r$ modulo $2^n-1$ has the lowest possible weight $2$. Using the results we obtained in the previous sections, we give an alternative proof and show additionally that (apart from sporadic cases for low values of $n$) these are the only cases where the inverses of Kasami exponents have weight $2$.

\begin{proposition}\label{prop:quad}
	Let $K_r$ be invertible modulo $2^n-1$ with $n\geq 6$ and $K_r^{-1}$ be the least positive residue of the inverse of $K_r$ modulo $2^n-1$. Then $\wt(K_r^{-1})=2$ if and only if $n=\frac{3r}{b}$ with $b|r$  and $\gcd(b,3)=1$. In these cases we have
	\begin{equation*}
		K_r^{-1} \equiv \begin{cases}
				2^{n-1}+2^{\frac{n}{3}-1}\pmod {2^n-1} & \text{if }b\equiv 1 \pmod 3 \\
				2^{n-1}+2^{\frac{2n}{3}-1} \pmod {2^n-1} & \text{if }b\equiv 2 \pmod 3 .
		\end{cases}
	\end{equation*}
\end{proposition}
\begin{proof}
	We go through the results in the earlier sections and check when $\wt(K_r^{-1})=2$ is fulfilled. In Proposition~\ref{prop:gcds1}, we have $\wt(K_r^{-1})=\frac{n-3d+4}{2}$ where $d=\gcd(r,n)$. We have $\frac{n-3d+4}{2}=2$ if and only if $n=3d$. So, $n=\frac{3r}{b}$ for some $b$ with $\gcd(b,3)=1$. We differentiate the two possible cases $e=1$ and $e=2$ corresponding to $b\equiv 1 \pmod 3$ and $b \equiv 2 \pmod 3$, respectively. If $e=1$, we are in Case (a) of Proposition~\ref{prop:gcds1} and the matrix $M_{a,r}$ looks as follows:
	\begin{equation*}
		M_{a,r}= \begin{pmatrix}
			0&0&0 \\
			\vdots & & \vdots \\
			0&0&0 \\
			1&1&0
		\end{pmatrix}.
	\end{equation*}
	Consequently, $K_r^{-1}\equiv 2^{\frac{n}{3}-1}+2^{\frac{n}{3}-1-r} \equiv 2^{n-1}+2^{\frac{n}{3}-1}\pmod {2^n-1}$. Here we used that $r\equiv \frac{n}{3} \pmod n$ since $b \equiv 1 \pmod 3$. If $e=2$, we apply the same procedure to $K_{n-r}$, so $K_{n-r}^{-1} \equiv 2^{\frac{n}{3}-1}+2^{\frac{n}{3}-1-(n-r)} \equiv  2^{n-1}+2^{\frac{n}{3}-1}\pmod {2^n-1}$ since here $r\equiv \frac{2n}{3} \pmod n$. Then $K_r^{-1} \equiv 2^{-2r}K_{n-r}^{-1}\equiv 2^{n-1}+2^{\frac{2n}{3}-1}\pmod {2^n-1}$.
	
	We now check Proposition~\ref{prop:gcds2}. In the Cases (a)-(d) we have $\wt(K_r^{-1})=\frac{n-d+2}{2}$, so $\wt(K_r^{-1})=2$ if and only if $d=n-2$. Since $d|n$ and $n>4$, this is not possible. 
	
	In the Cases (e) and (h) we have (using the notation from the proposition) $\wt(K_r^{-1})=\frac{n-d(s+1)+2}{2}$, so $\wt(K_r^{-1})=2$ if and only if $n-d(s+1)=2$. Since $d|n$, this implies $d|2$. Using the bound in Corollary \ref{cor:algdeg2}, we infer that $\wt(K_r^{-1})>2$ if $n\geq 6$. In the Cases (f) and (g) we have $\wt(K_r^{-1})=\frac{n-d(s+2)+2}{2}$. Again we get $d|2$ and the same argument as before yields $\wt(K_r^{-1})>2$.
	
	In Proposition~\ref{prop:gcds3} the inverses have always binary weight $\frac{n+2}{2}$, so no new cases are found.
\qed
\end{proof}
Note that the condition $n \geq 6$ is necessary. Indeed, for $n=5$ we get sporadic cases: Consider $K_2=13$ over $\F_{2^5}$. We have $\gcd(5,2)=1$ and $2\cdot 3 \equiv 1\pmod 5$, so $e=3$ and the inverse of $13$ modulo $2^5-1$ has weight $2$ by Corollary~\ref{cor:algdeg}.

\section{The Bracken-Leander exponent}
We now determine the inverse of the Bracken-Leander exponent $ BL_r=2^{2r}+2^r+1$ modulo $2^{4r}-1$ with $r$ odd. In this case, the exponent is not independent from the field size. Because of this, finding the inverse is much easier. We again use the modular add-with-carry approach. Theorem~\ref{thm:xiang_orig} applied to the Bracken-Leander exponents yields the following condition for the carry sequence.

\begin{theorem} \label{thm:BL}
	Let $r$ odd, $n=4r$, $a \in \{1,\dots,2^n-2\}$ and $BL_r$ be the Bracken-Leander exponent. We denote by $a=(a_{n-1},\dots,a_0)$ the binary expansion of $a$. 
	The following are equivalent:
	\begin{enumerate}[label=(\alph*)]
		\item $a$ is the inverse of $BL_r$ modulo $2^{n}-1$.
		\item There exists a carry sequence $c=(c_{n-1},\dots,c_{0})$ with $c_i \in \{0,1,2\}$ such that 
	\begin{align} 
		2c_0-c_{-1}+1&=a_{-2r}+a_{-r}+a_0 \label{eq:BL1}\\
		2c_i-c_{i-1}&=a_{i-2r}+a_{i-r}+a_i \text{ for all } i>0\label{eq:BL2}.
	\end{align} 
	Here, the indices are seen as elements in $\Z_{n}$. 
	\end{enumerate}
	The carry sequence in (b) is unique.
\end{theorem}
Observe that $\gcd(r,n)=r$ and $\frac{n}{\gcd(r,n)}=4$. The case here is thus similar to the $\frac{n}{\gcd(r,n)}$ even case of the Kasami functions. We again use $r$-matrices so that Eq.~\eqref{eq:BL1} and~\eqref{eq:BL2} have an easier structure. 

\begin{theorem} \label{thm:BL2}
	Let $r$ odd, $n=4r$, $a \in \{1,\dots,2^n-2\}$ and $BL_r=2^{2r}+2^r+1$ be the Bracken-Leander exponent. We denote by $a=(a_{n-1},\dots,a_0)$ the binary expansion of $a$.  Moreover, let \begin{equation*}
		M_{a,r}=\begin{pmatrix}
			a_{0,0}&a_{0,1}&a_{0,2}&a_{0,3} \\
			a_{1,0}&a_{1,1}&a_{1,2}&a_{1,3} \\
			\vdots& & &\vdots \\
			a_{r-1,0}&a_{r-1,1} &a_{r-1,2} & a_{r-1,3}
		\end{pmatrix}
	\end{equation*} be the $r$-matrix of $a$, i.e. $a\equiv \sum_{i=0}^{r-1}\sum_{j=0}^{3}a_{i,j}2^{i-jr} \pmod{2^n-1}$. 	The following are equivalent:
	\begin{enumerate}[label=(\alph*)]
		\item $a$ is the inverse of $K_r$ modulo $2^n-1$.
			\item There exists an $r$-matrix for the carry sequence $c$ of the form 
			\begin{equation*}
		M_{c,r}=\begin{pmatrix}
			c_{0,0}&c_{0,1}&c_{0,2}&c_{0,3} \\
			c_{1,0}&c_{1,1}&c_{1,2}&c_{1,3} \\
			\vdots& & &\vdots \\
			c_{r-1,0}&c_{r-1,1} &c_{r-1,2}& c_{d-1,3}
		\end{pmatrix}
	\end{equation*} 
			with $c_{i,j} \in \{0,1,2\}$  such that the following equations hold:
	\begin{align} 
		2c_{0,0}-c_{r-1,1}+1&=a_{0,2}+a_{0,1}+a_{0,0} \label{eq:rBL1} \\
		2c_{0,j}-c_{r-1,j+1}&=a_{0,j+2}+a_{0,j+1}+a_{0,j}  \text{ for }j\in \{1,2,3\} \label{eq:rBL2}\\
		2c_{i,j}-c_{i-1,j}&=a_{i,j+2}+a_{i,j+1}+a_{i,j} \text{ for all }i \in \{1,\dots,r-1\}, j \in \{0,1,2,3\}. \label{eq:rBL3}
	\end{align}
	\end{enumerate}
	The carry sequence (and thus its associated $r$-matrix) in (b) is unique.
\end{theorem}

It is easy to derive some strong necessary conditions from the equations. For example Eq.~\eqref{eq:rBL3} implies that, if $c_{i,j}=0$ for some $i>0$, then necessarily $c_{i-1,j}=a_{i,j+2}=a_{i,j+1}=a_{i,j}=0$, which inductively leads to $c_{i',j}=a_{i',j+2}=a_{i',j+1}=a_{i',j}=0$ for all $0<i'<i$. With some examples for small values of $n$, it is then quite easy to guess the correct $r$-matrices for the sequence $a$ and its associated carry sequence $c$.

\begin{proposition}
	Let $r$ odd, $n=4r$ and $BL_r=2^{2r}+2^r+1$ be the Bracken-Leander exponent. Then $BL_r^{-1}\equiv \sum_{i=0}^{r-1}\sum_{j=0}^{3}a_{i,j}2^{i-jr} \pmod{2^n-1}$ where the values $a_{i,j}$ are the entries of the matrix
	\begin{equation*}M_{a,r}=
		\begin{pmatrix}
			1&1&1&0 \\
			0&0&0&0 \\
			1&1&1&1 \\
			0&0&0&0 \\
			\vdots & & & \vdots\\
			1&1&1&1 \\
				0&0&0&0	\\
			1&1&1&1
		\end{pmatrix}.
	\end{equation*}
	We have $\wt(BL_r^{-1})=\frac{n+2}{2}$.
\end{proposition}

\begin{proof}
	The $r$-matrix of the corresponding carry sequence is
		\begin{equation*}(c_{i,j})=
		\begin{pmatrix}
			2&2&2&2 \\
			1&1&1&1 \\
			2&2&2&2 \\
			1&1&1&1  \\
			\vdots & & & \vdots\\
		2&2&2&2 \\
		1&1&1&1\\
			2&2&2&2
		\end{pmatrix}.
	\end{equation*}
	We verify Eq.~\eqref{eq:rBL1}-\eqref{eq:rBL3}. 
	Eq.~\eqref{eq:rBL1} holds because $c_{0,0}=c_{r-1,1}=2$ and $a_{0,0}=a_{0,1}=a_{0,2}=1$. Eq.~\eqref{eq:rBL2} holds because $c_{0,j}=c_{r-1,j+1}=2$ and $(a_{0,j},a_{0,j+1},a_{0,j+2})\in \{(1,1,0),(1,0,1),(0,1,1)\}$ if $j\in\{1,2,3\}$. 
	
	It only remains to check Eq.~\eqref{eq:rBL3}. For $i$ odd, we have $c_{i,j}=1$, $c_{i-1,j}=2$ and $a_{i,j}=a_{i,j+1}=a_{i,j+2}=0$. For $i>0$ even, we have $c_{i,j}=2$, $c_{i-1,j}=1$ and $a_{i,j}=a_{i,j+1}=a_{i,j+2}=1$, so Eq.~\eqref{eq:rBL3} is satisfied.
	
	To determine $\wt(BL_r^{-1})$, we count the number of ones in $M_{a,r}$, so $\wt(BL_r^{-1})=4\frac{r+1}{2}-1=\frac{n+2}{2}$.
	\qed
\end{proof}

\section{Conclusion}
In this paper, we introduced a new approach to find inverses of elements in $\Z_{2^n-1}$, using the modular add-with-carry approach. With this technique, we determined the inverse of all Gold exponents $G_r = 2^r+1$ and Kasami exponents $K_r=2^{2r}-2^r+1$ modulo $2^n-1$ (if they exist) as well as the inverse of the Bracken-Leander exponent $BL_r=2^{2r}+2^r+1$ modulo $2^{4r}-1$ with $r$ odd. With our contribution, the binary representations of the inverses of all known APN exponents as well as the inverses of all exponents that give rise to $4$-differentially uniform permutations in even dimension are found. The more general problem of inverting a given element $l$ in $\Z_{2^n-1}$ for all $n$ is still not well understood. It is a natural question if the approach using the modular add-with-carry algorithm can be generalized to other exponents. For every invertible $l$, we can find a defining set of equations for the binary representation of $l^{-1}$ and the corresponding carry sequence in the style of Eq.~\eqref{eq:fund_orig} in Theorem~\ref{thm:xiang_orig}. The difficulty then lies in finding the sequences that satisfy the equations. This has to be done on a case by case basis. 

Inversion in $\Z_{2^n-1}$ is not only interesting for questions relating to differential uniformity. For example, if $l$ is a \emph{complete permutation polynomial (CPP) exponent over} $\F_q$ (i.e. there exists an $a \in \F_q$ such that $ax^l$ and $ax^l+x$ are permutation polynomials), then also its inverse $l^{-1}$ modulo $q-1$ is a CPP exponent \cite{JAMS:NR82}. Several CPP exponents in even characteristic have been found (e.g. \cite{SIAM:CPP}, \cite{FFA:CCP}, \cite{FFA:CPP2}). For a complete classification of CPP exponents, finding explicit formulas for the corresponding inverses is an interesting research problem.

The modular add-with-carry approach can be easily modified to work also in the ring $\Z_{p^n-1}$ for a prime $p>2$ \cite[Theorem 4.1]{paryaddcarry}. In particular, it can be used to tackle the problem of inversion in $\Z_{p^n-1}$  (corresponding to inversion of monomials in odd characteristic). However, the equations in the style of Eq.~\eqref{eq:fund_orig} that have to be checked become more complicated. 

 \section*{Acknowledgments}
	I sincerely thank the reviewers for their careful reading of the paper, pointing out several typos and some helpful suggestions about the presentation of the results. Moreover, I would like to thank Gohar Kyureghyan for fruitful discussions and encouragement in the pursuit of this problem.
\bibliography{kasamiinverse_bib}
\end{document}